\documentclass{amsart}

\usepackage{amssymb, amsmath,  amsfonts} %, mathabx}
\usepackage{color}
\usepackage[all,cmtip]{xy}

\newtheorem{thm}{Theorem}[section]
\newtheorem{prop}[thm]{Proposition}
\newtheorem{lemma}[thm]{Lemma}
\newtheorem{cor}[thm]{Corollary}
\newtheorem{defn}[thm]{Definition}%[section]
\newtheorem{lem-defn}[thm]{Lemma/Definition}%[section]
\newtheorem{example}[thm]{Example}
\newtheorem{remark}[thm]{Remark}

\numberwithin{equation}{section}

 \numberwithin{equation}{section}

\newcommand{\C}{{\mathbb C}}
\newcommand{\Z}{{\mathbb Z}}
\renewcommand{\P}{{\mathbb P}}

\newcommand{\bull}{{\scriptscriptstyle \bullet}}
\newcommand{\la}{\lambda}

\newcommand{\cO}{{\mathcal O}}

\newcommand{\A}{{\lambda}}
\newcommand{\B}{{\mu}}
\newcommand{\sC}{{\nu}}
\newcommand{\sym}{{\mathpzc{S}}}

\newcommand{\comment}[1]{}
\newcommand{\type}{\text{type}}

\newcommand{\bigbrace}[2]{\genfrac{\{}{\}}{0pt}{}{#1}{#2}}
\renewcommand{\emptyset}{\varnothing}
\renewcommand{\tilde}{\widetilde}
\DeclareFontFamily{OT1}{pzc}{}
\DeclareFontShape{OT1}{pzc}{m}{it}%
              {<-> s * [0.900] pzcmi7t}{}
\DeclareMathAlphabet{\mathpzc}{OT1}{pzc}%
                                 {m}{it}
\DeclareMathAlphabet      {\mathsf}{OT1}{cmss}{m}{n}
\usepackage[mathscr]{eucal}
\usepackage{tikz}
\usetikzlibrary{matrix}
\usetikzlibrary{arrows}
\definecolor{darkgreen}{rgb}{0,0.5,0}
\usepackage{graphicx}
\makeatletter
\DeclareRobustCommand\widecheck[1]{{\mathpalette\@widecheck{#1}}}
\def\@widecheck#1#2{%
   \setbox\z@\hbox{\m@th$#1#2$}%
   \setbox\tw@\hbox{\m@th$#1%
      \widehat{%
         \vrule\@width\z@\@height\ht\z@
         \vrule\@height\z@\@width\wd\z@}$}%
   \dp\tw@-\ht\z@
   \@tempdima\ht\z@ \advance\@tempdima2\ht\tw@ \divide\@tempdima\thr@@
   \setbox\tw@\hbox{%
      \raise\@tempdima\hbox{\scalebox{1}[-1]{\lower\@tempdima\box\tw@}}}%
   {\ooalign{\box\tw@ \cr \box\z@}}}
\makeatother

\begin{document}{\allowdisplaybreaks[4]

\title[Equivariant Pieri rules for isotropic Grassmannians]{Equivariant Pieri rules for isotropic Grassmannians}

\author{Changzheng Li}
\address{Center for Geometry and Physics, Institute for Basic Science (IBS), Pohang 790-784, Republic of   Korea}
\email{czli@ibs.re.kr}

\author{Vijay Ravikumar}
\address{Chennai Mathematical Institute (CMI),
H1 SIPCOT IT Park, Siruseri, Kelambakkam, India}
%\curraddr{Department of Mathematics }
\email{vijayr@cmi.ac.in}
\thanks{ %The first author  is supported in part by %a RGC research grant from the Hong Kong Government.
%The second author  is  supported in part %by KRF-2007-341-C00006.
 }

\date{%January 1, 2001 and, in revised form, June 22, 2001.
      }


%\keywords{Lie superalgebra, spin  representations}

%\subtitle{}

\begin{abstract}
We give a Pieri rule for the torus-equivariant cohomology of (submaximal) Grassmannians of Lie types  $B$, $C$, and $D$.
%The non-equivariant limit of our formula induces the multiplication in the ordinary cohomology by the Chern classes of the tautological quotient bundle over the given Grassmannian.
To the authors' best knowledge, our rule is the first manifestly positive formula, beyond the equivariant Chevalley formula. We also give a simple proof of the equivariant Pieri rule for the ordinary (type $A$) Grassmannian.
\end{abstract}

\maketitle
\tableofcontents

\section{Introduction}\label{S:intro}

Let $V$ be an $N$-dimensional  complex vector space  equipped with a symmetric or skew-symmetric  bilinear form $\omega$, which can be either trivial or non-degenerate. The Grassmannians $IG_{\omega}(m, N)$ of classical Lie type parameterize $m$-dimensional isotropic vector subspaces of $V$. The cohomology ring of an isotropic Grassmannian $X=IG_{\omega}(m, N)$, or more generally of a homogeneous variety, has an additive basis of Schubert classes represented by Schubert subvarieties $X_{\A}$. One of the central problems of Schubert calculus is to find a manifestly positive formula for the structure constants of the cup product of two Schubert cohomology classes, or equivalently, for the triple intersection numbers of three Schubert subvarieties in general position. Such a positive formula, called a Littlewood-Richardson rule, has deep connections to various subjects, including geometry, combinatorics and representation theory.

%%{\color{red} mention exceptions for type d}
An isotropic Grassmannian $X$ can be written as a quotient of a classical complex simple Lie group $G$ by a maximal parabolic subgroup $P$ (with two notable exceptions of Lie type $D_n$, described in Section \ref{S:type_d}). Fix a choice of maximal complex torus $T$ and a Borel subgroup $B$ with $T\subset B\subset P$.   The Schubert varieties $X_{\lambda}$ (relative to $B$) are closures of $B$-orbits, and hence are $T$-stable. They give a basis $[X_{\A}]^T$ for the $T$-equivariant cohomology $H^*_T(X)$ as a $H^*_T(\mbox{pt})$-module.
The structure coefficients $N^{\sC}_{\A, \B}$ in the equivariant product,
\[ [X_{\A}]^T\cdot [X_{\B}]^T=\sum_{\sC}N^{\sC}_{\A, \B} [X_{\sC}]^T,\]
are homogeneous polynomials which satisfy a positivity condition conjectured by Peterson \cite{Peterson} and proved by Graham \cite{Grah}.  In particular, they are \emph{Graham-positive}, meaning they are polynomials in the negative simple roots,  with nonnegative integer coefficients.  These equivariant structure coefficients carry much more information than the triple intersection numbers of Schubert varieties, and are more challenging to study.
When the bilinear form $\omega$ is trivial, i.e., when $X=Gr(m, N)$ is a type $A$ Grassmannian, there has been extensive work on equivariant Littlewood-Richardson rules  \cite{KnutTao, Molev, Krei, ThYo}.
However, for Grassmannians of Lie type other than $A$,  there have been no  manifestly Graham-positive formulas, to the authors' best knowledge, except for the equivariant Chevalley formula which concerns multiplication by Schubert divisors  (\cite{koku,Brion}, see also e.g.   \cite[Theorem 11.1.7 (i)]{kumar}). We remark that an effective (but non-positive) algorithm for computing the structure coefficients for general $G/P$ is given in \cite{mih} (see also \cite{MoSa,KnutTao} for the type $A$ case).  Remarkably, a manifestly positive equivariant Littlewood-Richardson rule has recently been given by Buch  for two-step partial flag varieties of type $A$ \cite{Buch-equivTwostep}.

In the present paper, we give for the first time an \emph{equivariant Pieri rule} for Grassmannians of Lie types $B$, $C$, and $D$, as well as a new proof of the Pieri rule in type $A$.
Such a rule concerns products with the special Schubert classes $[X_p]^T$, which are related to the equivariant Chern classes of the tautological quotient bundle, and generate the $T$-equivariant cohomology ring.  Using geometric methods, we give a manifestly positive formula for the structure coefficients $N_{\A, p}^{\B}$ of the equivariant multiplication $[X_{\A}]^T\cdot[X_p]^T$.
For type $A$ Grassmannians $X=Gr(m, N)$, the equivariant Pieri rule has been even more extensively studied than the more general equivariant Littlewood-Richardson rule (see e.g. \cite{MoSa-Pieri,LaRaSa,Laksov,GaSa,Fun}).
Nevertheless,   we give a new  proof that  reveals an interesting reduction of arbitrary Pieri coefficients to much simpler ones.
Namely, we prove that any Pieri coefficient $N^\B_{\A,p}(Gr(m, N))$ is equal to a Pieri coefficient of the form $N^\sC_{\sC, p'}(Gr(m',N))$
in a possibly different Grassmannian $Gr(m', N)$. Such a reduction has also been made in \cite{huangli} and \cite{rob} in a combinatorial way, but our argument is much simpler and can explain its geometric origin.
Each    coefficient  $N^\sC_{\sC, p'}(Gr(m',N))$
is  the restriction of a special equivariant   Schubert class $[X_{p'}]^T$ to a $T$-fixed point of $Gr(m',N)$.
There have been several manifestly positive formulas for these \emph{restriction coefficients} (see \cite{koku, AJS, Billey} for general $G/P$; \cite{Ikeda,IkNa} for Lagrangian and maximal orthogonal Grassmannians; and for example \cite{buch_rimanyi} for type $A$ Grassmannians).
For  completeness, we include one more restriction formula in the appendix.
For isotropic Grassmannians $X=IG_{\omega}(m,N)$ of Lie types $B$ and $C$, we use geometric arguments to reduce  the Pieri coefficients $N^\B_{\A,p}(X)$ to sums of specializations of restriction coefficients $N^\sC_{\sC, p'}(Gr(m',N))$.
In the type $D$ case, we succeed in a similar way for most of the Pieri coefficients, and can reduce the rest to appropriate  restriction coefficients $N_{\sC, p'}^{\sC}(X')$ with respect to an isotropic Grassmannian $X'$ of type $D$.
We remark that in the case of the complete flag variety of Lie type $A$, an equivariant Pieri rule with respect to a distinct set of special Schubert classes is contained in  \cite{lamshi22}.
%We remark that there are various formulas for   restriction coefficients  $N^\sC_{\sC,p}(X)$ (see    \cite{Billey} for general homogeneous varieties $X=G/P$ or \cite{Ikeda, IkNa} for Lagrangian and maximal orthogonal Grassmannians).

To state our formula precisely, we will parametrize Schubert varieties by \textit{Schubert symbols} (also called \textit{index sets} in \cite{BKT2}, or \textit{jump sequences}), which gives a uniform description    for all classical Lie types. Schubert symbols for a Grassmannian   $IG_\omega(m, N)$ are subsets  $\A=\{\lambda_1<\lambda_2<\cdots<\lambda_m\}$ of the integer interval $[1,N]$ which in addition satisfy $\lambda_i+\lambda_j\neq N+1$ for all $i, j$ if $\omega$ is non-degenerate.
We denote by $|\A|$ the codimension of the Schubert variety $X_{\A}$ in $X$.
We also need the combinatorial relation $\A\to\B$ between Schubert symbols $\A$ and $\B$,
which says roughly that the cohomology class $[X_\B]$ occurs in some cohomological Pieri product involving $[X_{\A}]$.
Let us now restrict our attention to a Grassmannian of type $C_n$; i.e. we let $N=2n$  and $\omega$ be non-degenerate and  skew-symmetric. We adopt the notation $SG(m,2n)$ to refer to this symplectic Grassmannian.
Given $\A \to \B$, the pair $(\A, \B)$  defines two combinatorial sets $\mathcal{L}_{\A, \B}$ and $\mathcal{Q}_{\A, \B}$ (see Sections \ref{S:type_a} - \ref{S:type_d}  for precise descriptions), which index certain hyperplanes  and quadratic hypersurfaces in $\mathbb{P}^{2n-1}$, respectively.
We  adopt the Bourbaki \cite{Bour} expression for simple roots $\alpha_i$ (resp. $\hat\alpha_j$) of type $C_n$ (resp. $A_{2n-1}$) in terms of weights: $\alpha_n=2t_n$ and $\alpha_i=t_{i}-t_{i+1}$ for $1\leq i\leq n-1$ (resp. $\hat \alpha_j=\hat t_j-\hat t_{j+1}$ for $1\leq j\leq 2n-1$).  The inclusion of $T$ into a maximal complex torus of $GL(N,\C)\supset G$ induces a ring homomorphism
 $F: \Z[\hat t_1, \ldots, \hat t_{2n}] \to \Z[t_1, \ldots, t_{n}]$, defined by $F(\hat t_j)= t_j$ if $j\leq n$ and $F(\hat t_j)= -t_{2n+1-j}$ otherwise. Conveniently, the homomorphism $F$ sends simple roots of type $A_{2n-1}$ to simple roots of type $C_n$.
 It follows that $F$ sends Graham-positive polynomials of type $A_{2n-1}$ to Graham-positive polynomials of type $C_n$.  Using the specialization $F$, we therefore have

\begin{thm}[Equivariant Pieri rule for $SG(m, 2n)$]\label{T:intro}
Let $\A$ be a Schubert symbol for $SG(m, 2n)$ and   $1\leq p\leq 2n-m$. In   $H^*_T(SG(m, 2n))$,   we have
$$[X_{\A}]^T\cdot [X_{p}]^T=\sum N^{\B}_{\A, p} [X_{\B}]^T,$$
   where the sum is  over   Schubert symbols $\B$ satisfying $\A\to \B$ and {\upshape $|\B|\leq p + |\A|$}. Furthermore, each coefficient $N^{\B}_{\A,p}$ is a sum of $2^{\#\mathcal{Q}_{\A,\B}}$ specializations of restriction coefficients for the type $A$ Grassmannian $Gr(m', 2n)$:
\[N^\B_{\A,p}(SG(m,2n)) = \sum_{\mathcal{I}\subset \mathcal{Q}_{\A, \B}} F\left(N^{\sC_{\mathcal{I}}}_{\sC_{\mathcal{I}},p'}(Gr(m',2n))\right),\]
where $m' = m + |\B| - |\A|$, $p' = p + |\A| - |\B|$,
%where  $m'=2n-\#  \mathcal{L}_{\A, \B}- \# \mathcal{Q}_{\A, \B}$, $p'=m+p-m'$
and each $\sC_{\mathcal{I}}$ is an associated  Schubert symbol for $Gr(m',2n)$ (defined  explicitly in Section \ref{S:type_c}).
\end{thm}
\noindent In particular if $|\B|= p + |\A|$, then $p'=0$. As a consequence, the coefficient $N^\B_{\A,p}$ is a summation of   $2^{\# \mathcal{Q}_{\A, \B}}$  copies of the constant polynomial $1$. This reproduces  the ordinary Pieri rule of  Buch, Kresch and Tamvakis \cite{BKT2}.
For the odd orthogonal Grassmannian $OG(m, 2n+1)$ (i.e. the type $B$ case),   an equivariant Pieri coefficient  $N_{\A, p}^{\B}(OG(m, 2n+1))$ is generally not  a  multiple of another type $C$ equivariant Pieri coefficient, in contrast to the case of ordinary cohomology \cite{BilleyHai} (see also \cite[section 3.1]{BerSo}). Nevertheless,
we give a manifestly positive Pieri formula for it as well as for the type $D$ Grassmannian $OG(m, 2n)$.   We refer our readers to \textbf{Theorems \ref{T:type_b_reduction}} and  \textbf{\ref{T:type_d_reduction}} for the precise statements.

Our equivariant Pieri rules are obtained by geometric arguments, which include two major steps. Let us consider the    natural projections
 $$  IG_\omega(m, N)\overset{\pi}{\longleftarrow} IF_\omega(1, m; N)\overset{\psi}{\longrightarrow} IG_\omega(1,N),$$
 where $IF_\omega(1, m; N)$ denotes the corresponding two-step isotropic flag variety.
Let $Y_{\A, \B}$ denote the Richardson variety given by the intersection of the Schubert variety $X_\A$ with the opposite Schubert variety labeled by $\B$, let  $Z_{\A, \B}$ be the projected Richardson variety $\psi(\pi^{-1}(Y_{\A, \B}))$, and let $L_{p}$ be a subvariety of $IG_\omega(1,N) \subset \mathbb{P}^{N-1}$ with the property that $X_p = \pi(\psi^{-1}(L_{p}))$.
Finally, for any variety $Y$, let $\int^T_{Y}$ denote the equivariant pushforward along the morphism $Y \to \text{pt}$.
When $X$ is of type $A$ or $C$, the  natural  injection  $\iota: IG_\omega(1, N)\to \mathbb{P}^{N-1}$ is the  identity isomorphism,  and our first step is to write $N_{\A, p}^{\B}$ as the integral of an equivariant cohomology class in $\P^{N-1}$ via the projection formula:
 $$N_{\A, p}^{\B}=\int_X^T[Y_{\A, \B}]^T\cdot [X_p]^T=\int_{\P^{N-1}}^T[Z_{\A, \B}]^T\cdot [L_p]^T.$$
When $X$ is of   type $B$  or $D$, the injection $\iota$ is no longer surjective, but a more  involved analysis still works.
Such a technique has led to a Pieri rule for the ordinary cohomology of isotropic Grassmannians \cite{Sott, BKT2} as well as  triple intersection formulas in K-theoretic Schubert calculus \cite{buchMihalcea,buch_ravikumar, Ravi}. In equivariant cohomology, it  reduces $N^\B_{\A,p}$ to an easier calculation in $H^*_T(\mathbb{P}^{N-1})$. However, any direct computation of $\int_{\P^{N-1}}^T[Z_{\A, \B}]^T\cdot [L_p]^T$ involves sign cancelations, and fails to be manifestly positive.

Our key observation is that the projected Richardson variety $Z_{\A,\B} \subset \P^{N-1}$ can be degenerated into $2^{\#\mathcal{Q}_{\A,\B}}$ linear subvarieties $Z_{\sC_{\mathcal{I}},\sC_{\mathcal{I}}}$, indexed by subsets $\mathcal{I} \subset \mathcal{Q}_{\A,\B}$.  These in turn can be interpreted as projections to $\P^{N-1}$ of Richardson varieties $Y_{\sC_{\mathcal{I}}, \sC_{\mathcal{I}}}$ in a type $A$ Grassmannian $X' := Gr(m',N)$.  Applying the projection formula in the reverse direction, we reduce $\int_{\P^{N-1}}^T[Z_{\A, \B}]^T\cdot [L_p]^T$ to a sum of quantities that are easy to compute, with positivity apparent:
$$\int_{\P^{N-1}}^T[Z_{\A, \B}]^T\cdot [L_{p}]^T=\sum_{\sC_{\mathcal{I}}}\int_{\P^{N-1}}^T[Z_{\sC_{\mathcal{I}}, \sC_{\mathcal{I}}}]^T\cdot [L_{p'}]^T=\sum_{\sC_{\mathcal{I}}}\int_{X'}^T[Y_{\sC_{\mathcal{I}}, \sC_{\mathcal{I}}}]^T\cdot [X_{p'}]^T.$$
In particular, each subvariety $Y_{\sC_{\mathcal{I}}, \sC_{\mathcal{I}}}$
is simply the $T$-fixed point in $X'$ corresponding to the Schubert symbol $\sC_{\mathcal{I}}$, and $\int^T_{X'}[Y_{\sC_{\mathcal{I}}, \sC_{\mathcal{I}}}]^T\cdot [X_{p'}]^T$
is the (specialized) restriction of $[X_{p'}]^T$ to that $T$-fixed point.

In addition to the equivariant Pieri rules, there is another important component to the equivariant Schubert calculus for isotropic Grassmannians,  namely the equivariant Giambelli formulas which express an arbitrary Schubert class as a polynomial in special equivariant Schubert classes or Chern classes (see \cite{LaRaSa,Mihalcea_Giambelli,  IkNa,Wilson,IMN, IkMa, Tamv, TamvWilson} and references therein).
The torus-equivariant cohomology ring of an isotropic Grassmannian  (or more generally of a homogeneous variety) behaves more simply than its ordinary cohomology ring, in the sense that it is essentially determined by the equivariant Chevalley formula of multiplication by divisor classes due to Mihalcea's criterion \cite{mih}. This has led to nice applications on the Giambelli-type formula for type $A$ flag varieties by Lam and Shimozono \cite{lamshi33}. However, it is far from obvious   how one can obtain applications to full Pieri-type formulas from this criterion or the Giambelli-type formulas.

The paper is organized as follows. In section 2, we introduce some basic notions for Grassmannians across classical Lie types. In section 3, we review basic properties of the torus-equivariant cohomology. In sections 4-7, we give the equivariant Pieri rules for Grassmannians of type $A$, $C$, $B$ and $D$ respectively. Finally in the appendix, we include a manifestly positive formula for the restriction coefficients for type $A$ Grassmannians.

\subsection*{Acknowledgements}
The authors thank Hongjia Chen, Thomas Hudson, K. N. Raghavan, Sushmita Venugopalan, and especially Anders Skovsted Buch  and Leonardo Constantin Mihalcea   for useful discussions and helpful feedback. The authors are also grateful to the referee for the careful reading and valuable comments. The first author is supported by IBS-R003-D1.

\section{Grassmannians of classical Lie types}\label{S:grassmannians}

Let $V$ be an $N$-dimensional  complex vector space  equipped with a   bilinear form $\omega$, and denote
\[
IG_\omega(m, N) := \{\Sigma \leqslant V: \dim_\C \Sigma=m,\,\, \omega(\+v,\+w) = 0 \hphantom{.} \forall \+v,\+w \in \Sigma\}.
\]
Throughout this paper, we will  consider a Grassmannian variety $X=IG_\omega(m, N)$  of Lie type $A_{n-1}$, $B_n$,  $C_n$ or $D_n$, characterized by the following properties and notations respectively:
\begin{enumerate}
  \item[$A_{n-1}$:]  $\omega(\cdot, \cdot)\equiv 0$. Namely $X=Gr(m, n)$ is an ordinary Grassmannian, where $N=n$.
  \item[$B_n$:]  $\omega$  is non-degenerate and symmetric, and $N$ is odd. Then $X=OG(m, 2n+1)$  is called an odd orthogonal Grassmannian, where $N = 2n+1$.
  \item[$C_n$:]  $\omega$  is non-degenerate and skew-symmetric, and $N$ is even. Then $X=SG(m, 2n)$ is called a symplectic Grassmannian, where $N=2n$.
  \item[$D_n$:]  $\omega$  is non-degenerate and symmetric, and $N$ is even. Then $X=OG(m, 2n)$  is called an  even orthogonal Grassmannian, where $N=2n$.
\end{enumerate}
In all of these cases, we will assume $m \leq n$.
When $m=n$ we refer to $X$ as a Lagrangian Grassmannian in type $C_n$ and a maximal orthogonal Grassmannian in types $B_n$ and $D_n$.
%%For $X=OG(m, 2n)$,   we will further assume $m<n$, since  $OG(m, 2n)$  consists of two connected components, each of which is isomorphic to $OG(n-1, 2n-1)$.
The Grassmannian $X$ is a smooth projective variety of complex dimension $m(n-m)$ in the type $A_{n-1}$ case,   $2m(n-m)+\frac{m(m+1)}{2}$ in types $B_n$ and $C_n$, and $2m(n-m)+\frac{m(m-1)}{2}$ in type $D_n$.

 Take an isomorphism $V\cong \C^N$ by specifying a  basis $\{\mathbf{e}_1, \cdots, \mathbf{e}_N\}$ of $V$ which in addition satisfies $\omega(\mathbf{e}_i, \mathbf{e}_j)=\delta_{i+j, N+1}$ for all $i\leq j$ if $\omega$ is non-degenerate.   Define a complete (isotropic) flag $E_{\bull}$
by $E_j := \langle \+e_1,\ldots,\+e_j \rangle$, the span of the first $j$ basis vectors.  Let $[1, N]$ denote the set of integers $\{1, 2, \cdots, N\}$. A Schubert symbol for $X$ is a subset $\A=\{\lambda_1<\lambda_2<\cdots<\lambda_m\}$ of $[1, N]$ which in addition satisfies $\lambda_i+\lambda_j\neq N+1$ for all $i\leq j$ if $\omega$ is non-degenerate.  The set of Schubert symbols for $X$ is denoted by $\mathfrak{S}(X)$.

\comment{
There is a one-to-one correspondence   $\mathfrak{S}(SG(m, 2n))\overset{\simeq}{\longrightarrow}\mathfrak{S}(OG(m, 2n+1))$, defined by
\[
\A \mapsto \A^+ := (\A \cap [1,n]) \cup \{a \in [n+2,2n+1]: a-1 \in \A\}.
\]
}

The Schubert subvarieties of $X$ (relative to $E_\bull$) are parameterized by Schubert symbols as follows (see \cite{BKT2,leungli44} for identifications with alternate parametrizations).
The Grassmannian $X=IG_\omega(m, V)$ admits a transitive action of an appropriate reductive complex Lie group  $G$ of rank $n$ (with the exception of $X = OG(n,2n)$, where the action is transitive only when restricted to one of the two connected components). Precisely,   $G= GL(n,\C), SO(2n+1, \C)$,  $Sp(2n, \C)$ or $SO(2n, \C)$, according to whether $X$ is of type $A_{n-1}$, $B_n$,  $C_n$ or $D_n$.  The stabilizer of  $E_{\bull}$ in the $G$-action is a Borel subgroup  $B$ of $G$, which contains a maximal complex torus $T\cong (\C^*)^n$ with eigenvectors $\mathbf{e}_1, \cdots, \mathbf{e}_N$. Each Schubert symbol $\A$ indexes a $T$-fixed point
$\Sigma_\A := \langle \+e_{\lambda_1},\ldots,\+e_{\lambda_m} \rangle$,  whose $B$-orbit closure is the Schubert variety $X_\A := \overline{B.\Sigma_\A}$. Now we let $B^-$ denote  the opposite Borel subgroup intersecting $B$ at $T$, and define the opposite Schubert
variety $X^\A := \overline{B^-.\Sigma_\A}$.
 We denote the codimension of $X_\A$ in $X$ (equivalently the dimension of $X^\A$) by $|\A|$, which is given by the formula:
 \begin{equation*}
  |\A|=\begin{cases}
    \dim X- \sum^m_{j=1}(\lambda_j - j)&\mbox{for type } A_{n-1},\\
    &\\
    \dim X-\sum^m_{j=1}(\lambda_j - j - \#\{i<j:\lambda_i+\lambda_j>N+1\})&\mbox{for type }   C_n,\\
    &\\
    \dim X-\sum^m_{j=1}(\lambda_j - j - \#\{i \leq j:\lambda_i+\lambda_j>N+1\})&\mbox{for types }   B_n, D_n.
 \end{cases}
 \end{equation*}

We note that in types $A_{n-1}$, $B_n$ and $C_n$ the Schubert varieties $X_\A$ and $X^\A$ have alternative characterizations by the \emph{Schubert conditions}:
\begin{align*}
 X_\A &= \{\Sigma \in X: \dim(\Sigma \cap E_{\lambda_j}) \geq j \text{ for } 1 \leq j \leq m\}, \text{ and} \\
 X^\A &= \{\Sigma \in X: \dim(\Sigma \cap \langle \+e_{\lambda_j},\+e_{\lambda_j+1} \ldots, \+e_{N}\rangle) \geq m+1-j \text{ for } 1 \leq j \leq m\}.
\end{align*}
For an analagous characterization in type $D_n$, see \cite[Proposition A.2]{BKT4}.

  Among  the Schubert varieties $X_\lambda$, there are special Schubert varieties $X_{\sym_p}$, also denoted simply as $X_p$, which have codimension $p$ and are determined by the single condition $\{\Sigma \in X: \dim(\Sigma \cap E_{n_p})\geq 1\}$.
Here $p$ is a positive integer with $p \leq N-m$ if $X$ is of type $A$ or $C$, and $p\leq N-m-1$ otherwise. The special  Schubert symbol   $\sym_p$, together with the integer  $n_p$ will be  specified  when we discuss  the equivariant Pieri rules individually later.

Given Schubert symbols $\A = \{\lambda_1 < \ldots < \lambda_m\}$
and $\B = \{\mu_1 < \ldots < \mu_m\}$, we write
$\A \leq \B$ if $\lambda_j \leq \mu_j$ for $1 \leq j \leq m$.  In types $B$ and $C$,
we have $\A \leq \B$ if and only if $X_\A \subset X_\B$.  In Section \ref{S:type_d}, we define a stronger relation $\preceq$ on Schubert symbols which coincides with the relation $X_\A \subset X_\B$ in type $D$.
The \emph{Richardson variety} $Y_{\A,\B} := X_\A \cap X^\B$ has dimension $|\B| - |\A|$, and
is nonempty if and only if $\B \leq \A$ (resp. $\B \preceq \A$ in type $D$).
In particular,  $Y_{\A, \A}$ consists of a single $T$-fixed point $\Sigma_\A$.

Given $\B \leq \A$, we define the associated \emph{Richardson diagram} by
$$D_X(\A,\B) := \{(j,c) : \mu_j \leq c \leq \lambda_j \} \subset [1,m] \times [1,N].$$
We will simply denote $D_X(\A,\B)$ as $D(\A, \B)$ whenever there is no confusion, and will  represent this set visually as an $m \times N$ matrix with stars for
every entry in $D(\A,\B)$ and zeros elsewhere.
We say $c \in [1,N]$ is a \emph{zero column} of $D(\A,\B)$ if $\lambda_j < c < \mu_{j+1}$ for some $j \in [0,m]$,
where we set $\lambda_0=0$ and $\mu_{m+1} = N+1$ for convenience. In types $B$, $C$, and $D$, we will also define the notion of a \emph{cut} in $D(\A,\B)$. As we will see,  zero columns and cuts will be further used to define combinatorial sets $\mathcal{L}_{\A,\B}$ and  $\mathcal{Q}_{\A,\B}$,  which index certain hyperplanes  and quadratic hypersurfaces in $\mathbb{P}^{N-1}$.

\section{Torus Equivariant Cohomology}\label{S:cohomology}
The aim of this paper is to give an equivariant Pieri rule for $X=IG_\omega(m, N)$. In this section we  review some basic properties of the torus-equivariant cohomology, and do the first step of our reductions.
\subsection{Basic properties of $H_T^*(X)$} We refer the readers to \cite{fuand, kumar} and references therein for   the facts mentioned here.
Recall that $X$ has a transitive $G$-action, and $T\subset G$ is a fixed maximal complex torus.
Let pt denote a point equipped with a trivial $T$-action.
The $T$-equivariant cohomology $H^*_T(\text{pt})$ %, up to a canonical isomorphism,
is given by
 $$H^*_T(\text{pt}) = \Lambda := \Z[t_1,\ldots,t_n],$$
 where  each $t_i$ is defined as follows.
Let $\chi_i$ denote the character  that sends   $(z_1,\ldots,z_n) \in T = (\C^*)^n$ to $z_i$.
Note that $\chi_i$ induces a  one-dimensional representation $\C_{\chi_i}$ of $T$, defined by
$(z_1,\ldots,z_n).v \mapsto z_iv$ for any $v \in \C$.
We then define $t_i := c^T_1(\C_{\chi_i})$, for $1 \leq i \leq n$, where
$\C_{\chi_i}$ is treated as a $T$-equivariant line bundle over the point pt.\footnote{It is also common to use the basis of opposite Schubert varieties together with
the identification $t_i := -c^T_1(\C_{\chi_i})$.  The resulting structure coefficients
will still be Graham-positive (in the sense of being positive polynomials in the negative simple roots).}

The  $T$-equivariant cohomology  $H^*_T(\cdot)$ is a contravariant functor from complex $T$-spaces to graded $\Lambda$-algebras.
For a $T$-invariant subvariety  $Y$ of $X$, we denote the natural inclusion and projection, respectively, by
 $$\iota_Y: Y\hookrightarrow X;\qquad \rho_Y: Y\longrightarrow \{\mbox{pt}\}.$$
 Both $\iota_Y$ and $\rho_Y$ are proper maps, and hence induce
 equivariant pushforwards $\iota_{Y,*}: H^*_T(Y) \to H^*_T(X)$ and
 $\rho_{Y,*}: H^*_T(Y) \to H^*_T(\text{pt}) = \Lambda$  respectively.
 We will henceforth denote the map $\rho_{Y,*}$ by $\int^T_Y$, as in the introduction.
The subvariety $Y$ determines an equivariant cohomology class $[Y]^T \in H^{2\text{codim}Y}_T(X)$
 under $\iota_{Y,*}$.

We notice that the Schubert subvarieties $X_\A$ and $X^\A$ are all $T$-invariant. In fact,
$H^*_T(X)$   is a  free $\Lambda$-module and the sets
    $\{[X_\A]^T\}$ and $\{[X^\A]^T\}$ form $\Lambda$-bases for  $H^*_T(X)$.
    These bases are dual with respect to the equivariant pushforward to the
    point:
        $$\int^T_X [X_\A]^T \cdot [X^\B]^T = \delta_{\A,\B}.$$
The structure coefficients in the equivariant product,
\[ [X_{\A}]^T\cdot [X_{\B}]^T=\sum_{\sC}N^{\sC}_{\A, \B}(X) [X_{\sC}]^T,\]
are simply denoted $N^{\sC}_{\A, \B}=N^{\sC}_{\A, \B}(X)$ when there is no confusion, and are given by
 $$N_{\A,\B}^{\sC}= \int^T_X [X_\A]^T \cdot [X_\B]^T\cdot [X^{\sC}]^T.$$
They    are  homogeneous polynomials of degree $(|\A|+|\B|-|\sC|)$ in the negative simple roots with non-negative integer coefficients \cite{Grah} (we will specify the simple roots in terms of the weights $t_i$ later).

Let $X^T$ denote the set of $T$-fixed points $\Sigma_\A$ in $X$. The $T$-equivariant inclusion $\iota_T:=\iota_{X^T}$ induces an injective ring morphism
 $$\iota_T^*: H^*_T(X) \hookrightarrow H^*_T(X^T) =\oplus_{\A} \Lambda,$$
  which extends to an isomorphism over the fraction field of $\Lambda$.
We call    $N^\nu_{\nu,\B}$ a \textit{restriction coefficient}, due to the following well-known fact \cite{Arabia}:
  %%\begin{prop}\label{L:pieri_as_restriction}
%%Given $\B$ and $\nu \in \mathfrak{S}(X)$, we have
\[N^\nu_{\nu,\B}  = \iota^*_{\Sigma_\nu}[X_{\B}]^T.\]
We have $N^\nu_{\nu,\B}=0$  unless $\Sigma_\nu \in X_\B$, or equivalently $\nu\leq  \B$.
%%\end{prop}
This vanishing is a special case of the fact that $N_{\A, \B}^\sC=0$ unless $\sC\leq  \A$ and $\sC\leq  \B$ \cite{koku}.

\subsection{First step  for the equivariant Pieri rule}Our goal is to give an equivariant Pieri rule for $X$; that is, a manifestly positive formula for equivariant multiplication by a special Schubert class,
 \[ [X_{\A}]^T\cdot [X_p]^T=\sum_{\B}N^{\B}_{\A, p}  [X_{\B}]^T.\]
 We remark that the Grassmannian $IG_{\omega}(m, N)$ carries an exact sequence of tautological bundles: $0\to \mathcal{S}\to \C^{N}\to \mathcal{Q}\to 0$.
 The special equivariant Schubert class $[X_p]^T$ coincides with
the  relative equivariant Chern class $c_p^T(\mathcal{Q}-E_{n_p})$  (possibly up to a factor of $2$ in types $B$ and $D$), where  $E_{n_p}$ is regarded as a trivial vector subbundle of $\C^N$ which is stable under the action of $T$.

 We let $IF_\omega(1, m; N)$ denote the two-step isotropic flag variety, and consider the natural projections $\pi$ and $\psi$ as below.
\begin{center}
\begin{tikzpicture}[node distance = 1.8cm]
  \node (OF) {$IF_\omega(1,m;N)$};
  \node(dummy) [right of=OF]{};
  \node (OG) [below of=OF] {$X=IG_\omega(m, N)$};
  \node (Q) [right of=dummy] {$IG_\omega(1, N) \subset \P^{N-1}$};
  %\node (P) [right of=Q] {$\P^{N-1}$};
  \draw[->] (OF) edge node[left]{$\pi$} (OG);
  \draw[->] (OF) edge node[above]{$\psi$}(Q);
  %\draw[right hook->] (Q) edge node[above]{$\iota$}(P);
  \end{tikzpicture}
\end{center}

\noindent Let $Z:=IG_\omega(1,N) \subset \P^{N-1}$.  Recall that $X_p = \{\Sigma \in X: \dim(\Sigma \cap E_{n_p})\geq 1\}$.
It follows that  $X_p = \pi\big(\psi^{-1}(Z \cap \P(E_{n_p}))\big)$.  Note that the subvariety $L_p \subset Z$ mentioned in the
introduction is precisely $Z \cap \P(E_{n_p})$.  Let $Z_{\A,\B}$ be the projected Richardson variety $\psi(\pi^{-1}(Y_{\A,\B}))$. Following \cite[\S 5]{BKT2}, we write
$\A \to \B$ if  appropriate   combinatorial relations are  satisfied by the Schubert symbols $\A$ and $\B$.  Precise descriptions of the relation $\A \to \B$ and the varieties $\P(E_{n_p})$ and $Z_{\A,\B}$  will be postponed to the type-dependent discussions   in the next four sections.
As mentioned in the introduction, we will achieve our equivariant Pieri rule in two main steps. Let us end this section by   the first step:
\begin{prop}\label{P:move_to_projective_space}
 Given $\B \leq \A \in \mathfrak{S}(X)$ and $p \in [1,N-m]$ (resp. $[1,N-m-1]$ in type $B$ or $D$), we have $N^\B_{\A,p}(X) = 0$ unless
$\A \to \B$, $|\B| \leq |\A| + p$, and $\B \leq \sym_p$. When $\A \to \B$, we have
\begin{align*}
N_{\A, p}^\B
&=\int^T_Z[Z_{\A,\B}]^T \cdot [Z \cap \P(E_{n_p})]^T.
\end{align*}
\end{prop}

\begin{proof}
It is well known that $N^{\B}_{\A,p}$ vanishes unless $|\B| \leq |\A| + p$ and $\B \leq \sym_p$ \cite{koku}.
Given Schubert symbols $\B \leq \A$, we can apply the projection formula to get
\[
 N^{\B}_{\A,p} = \int^T_X[Y_{\A,\B}]^T \cdot [X_p]^T = \int^T_Z\psi_*[\pi^{-1}(Y_{\A,\B})]^T \cdot [Z \cap \P(E_{n_p})]^T.
\]
In \cite[\S 5]{BKT2} it is shown that the projection $\psi:\pi^{-1}(Y_{\A,\B}) \to Z_{\A,\B}$
has positive dimensional fibers when $\A \not\to \B$, and is a birational isomorphism when $\A \to \B$.
The Gysin pushforward $\psi_*[\pi^{-1}(Y_{\A,\B})]^T$ therefore vanishes unless $\A \to \B$,
in which case it equals $[Z_{\A,\B}]^T$.
\end{proof}

%%%  \begin{lemma}\label{L:vanishing} $N^\B_{\A,p}(X) = 0$ unless $\A \to \B$, $|\B| \leq |\A| + p$, and $\B \leq \sym_p$.  \end{lemma}

%Unless otherwise stated, we will always start with $\mu\leqslant \la$, so that the Richardson diagram $D(\la, \mu)$ is well defined. As shown in \cite{BKT2, Ravi}, the projected Richardson variety  $Z_{\la,\mu}$ is a complete intersection in $\mathbb{P}^{N-1}$ cut out by linear and quadratic equations. The defining linear (resp. quadratic) equations of $Z_{\la, \mu}$ are indexed by a subset $\mathcal{L}_{\A, \B}$ (resp. $\mathcal{Q}_{\A, \B}$) of $[1, N]$, which will be specified in the type-dependent discussions. Here we introduce a notation for later use.
%\begin{defn} Given   $\mu\leqslant \lambda\in \mathfrak{S}(X)$, we define an associated set
%   $$\nu(\lambda, \mu):=[1, N]\setminus \mathcal{L}_{\A, \B}.$$
%\end{defn}

\section{Type $A$ Pieri Reduction}\label{S:type_a}
In this section, we consider the Grassmannian  $X=Gr(m,n)$.
We will show that each equivariant Pieri coefficient $N^\B_{\A,p}(X)$
is equal to  a restriction coefficient $N_{\sC, p'}^\sC(X')$ for  a possibly
different type $A$ Grassmannian $X'$.
By a result of Graham \cite{Grah}, each coefficient is a homogeneous polynomial in $\mathbb{Z}_{\geq0}[-\hat{\alpha}_1, \cdots, -\hat{\alpha}_{n-1}]$.
Here $\hat{\alpha}_i:=t_i-t_{i+1}\in H_T^2(\mbox{pt})$ can be naturally identified with the simple roots of $SL(n, \C)\subset GL(n, \C)=G$. More precisely, each character $\chi_i: T\to \mathbb{C}^*$ induces a weight $\epsilon_i: \mbox{Lie}(T) \to \mathbb{C}$, and $\hat{\alpha}_i$ is identified with the simple root $\epsilon_i-\epsilon_{i+1}$.  We will use an analagous identification in types $B$, $C$, and $D$ without further comment.
\comment{
The Schubert symbols $\lambda=\{\lambda_1<\cdots<\lambda_m\}$ can be identified with partitions $(a_1, \cdots, a_m)$ by $a_i:=n+i-m-\lambda_i$.
}

The special Schubert varieties $X_p=X_{\sym_p}$ are indexed by integers $1\leq p\leq n-m$, and satisfy
$X_p = \pi(\psi^{-1}(\P(E_{n+1-p-m})))$. The special Schubert symbol $\sym_p$ is given by
  \[\sym_p=\{n+1-m-p, n+2-m, \cdots, n\},\]
which corresponds to the special partition $(p, 0, \cdots, 0)$ in the traditional parameterization of Schubert varieties by partitions.

Given Schubert symbols $\mu\leqslant \la$, we let $x_1,\ldots,x_n$ denote the basis of  $V^*$ dual to $\+e_1,\ldots,\+e_n$. The projected Richardson variety
 $Z_{\A,\B} =\psi(\pi^{-1}(Y_{\A, \B}))\subset \P^{n-1}$ is defined by the linear equations
 $x_c=0$ with $c$ varying over the set
\[
\mathcal{L}_{\A,\B} := \{c \in [1,n]: c \text{ is a zero column of } D(\A,\B)\}
\]
(see e.g. \cite[\S 9.4]{fulton} or \cite[Lemma 3.3]{buch_ravikumar}).
We define an associated set
      $$\nu:=\nu(\lambda, \mu) = [1,n] \setminus \mathcal{L}_{\A,\B}$$
      which is simply the complement of $\mathcal{L}_{\A,\B}$.
      The set $\nu$ can naturally be thought of  as a Schubert symbol for the Grassmannian
$X' := Gr(m',n)$ with  $m' := \#\sC$.

We consider the natural projections
\begin{align}\label{E:projections-new}
    Gr(m', n)\overset{\pi'}{\longleftarrow} F\ell(1, m'; n)\overset{\psi'}{\longrightarrow} Gr(1,n)=\mathbb{P}^{n-1},
\end{align}
 and study the projected Richardson variety  $Z_{\sC,\sC}:=\psi'\circ (\pi')^{-1}(X'_{\sC}\cap X'^{\sC})\subset\mathbb{P}^{n-1}$. We have the following trivial but important observation.
\begin{lemma}\label{L:same_zero_columns}
The projected Richardson varieties $Z_{\A,\B}$ and $Z_{\sC,\sC}$ in $\P^{n-1}$
are equal.
\end{lemma}
\begin{proof}
The diagrams $D_X(\A,\B)$ and $D_{X'}(\sC,\sC)$ have the same zero columns, and hence
the same equations define $Z_{\A,\B}$ and $Z_{\sC,\sC}$.
\end{proof}

Given Schubert symbols  $\A=\{\lambda_1 < \ldots < \lambda_m\}$ and $\B = \{\mu_1 < \ldots < \mu_m\}$  in $\mathfrak{S}(X)$, we write $\A \to \B$, if both of the following hold: (1) $\B \leq \A$;   (2)   $\lambda_i < \mu_{i+1}$ for $1 \leq i \leq m-1$.
In other words the Richardson diagram  $D(\A,\B)$ is well defined, and none of its columns have more than one star\footnote{Equivalently, letting $P(\lambda)$ and $P(\mu)$ be the corresponding \emph{partitions}, the skew-diagram $P(\mu)/P(\lambda)$ is a horizontal strip.}.

\begin{prop}[Equivariant Pieri rule for $Gr(m,n)$]\label{P:redux_A}
Given Schubert symbols $\A \to \B$ and a positive integer $p \leq n-m$
such that $|\B| \leq |\A| + p$, we define $p':=|\A|+p-|\B|\geq 0$. We then
have
\[N^\B_{\A,p}(X) = N^\sC_{\sC,p'}(X').\]
Furthermore if  $\B \leq \sym_p$, then
$N^\B_{\A,p}(X) \neq 0$.
\end{prop}
\begin{proof}Since $\A \to \B$, we have $\dim(Y_{\A,\B}) = \#\sC - m
= m'-m$.  On the other hand,  $\dim(Y_{\A,\B})=|\B| - |\A|$. It follows that
$m'-m=p-p'$.
Therefore we have
\begin{align*}
 \int^T_X [X_\A]^T \cdot [X^\B]^T \cdot [X_{p}]^T
&=\int^T_{\P^{n-1}} [Z_{\A,\B}]^T \cdot [\P(E_{n+1-p-m})]^T\\
&=\int^T_{\P^{n-1}} [Z_{\sC,\sC}]^T \cdot [\P(E_{n+1-p'-m'})]^T\\
&=\int^T_{X'} [X'_\sC]^T \cdot [(X')^\sC]^T \cdot [X'_{p'}]^T.
\end{align*}
The first and third equalities follow
from Proposition \ref{P:move_to_projective_space}, and
the second equality follows from Lemma \ref{L:same_zero_columns} and the equality
$m'+p'=m+p$.

The statement of nonvanishing follows from the
fact that $\B \leq \sym_p$ if and only if
$\mu_1 \leq n+1-m-p$, and
$\sC \leq \sym_{p'}$ if and only if
$\nu_1 \leq n+1-m'-p'$.  Since
$\nu_1 = \mu_1$ and $-m-p = -m' - p'$,
these statements are equivalent.
It is well known that
$N^\sC_{\sC,p'} \neq 0$ exactly when
$\sC \leq \sym_{p'}$ (see e.g. \cite{KnutTao}).
%It is easy to see that $\B \leq \sym_p$ in $\mathfrak{S}(X)$ if and only if $\sC \leq \sym_{p'}$ in $\mathfrak{S}(X')$, equivalently,  $N^\sC_{\sC,p'}(X') \neq 0$ (see e.g. \cite{KnutTao}). Hence,   the statement of nonvanishing follows.
\end{proof}
 There have been several manifestly positive formulas for the restriction coefficients $N^\sC_{\sC,p'}(X')$
\cite{koku, AJS, Billey, buch_rimanyi}, and we include one in Appendix \ref{S:appendix} that uses Schubert symbols.
Combining Lemma \ref{L:restriction_formula2} with Proposition \ref{P:redux_A} yields an equivariant Pieri rule for the ordinary Grassmannian.

\begin{example}
 Let $X = Gr(3,8)$ and let $\A = \{1,4,8\}$ and $\B = \{1,3,6\}$.
 The Richardson diagram $D(\A,\B)$ is then
 \[
\left(
  \begin{array}{cccccccc}
    * & 0 & 0 & 0 & 0 & 0 & 0 & 0  \\
    0 & 0 & * & * & 0 & 0 & 0 & 0  \\
    0 & 0 & 0 & 0 & 0 & * & * & *  \\
  \end{array}
\right).
\]
The associated set $\nu:=\sC(\lambda, \mu)=\{1,3,4,6,7,8\}$ is a Schubert symbol for $X'=Gr(6,8)$. We have
$N^\B_{\A,5}(X) = N^\sC_{\sC,2}(X') = (t_2-t_1)(t_5-t_1),$ where the
final equality follows from Lemma \ref{L:restriction_formula2}.

\comment{and the diagram $D(\sC,\sC)$ will be
 \[
\left(
  \begin{array}{cccccccc}
    * & 0 & 0 & 0 & 0 & 0 & 0 & 0  \\
    0 & 0 & * & 0 & 0 & 0 & 0 & 0  \\
    0 & 0 & 0 & * & 0 & 0 & 0 & 0  \\
    0 & 0 & 0 & 0 & 0 & * & 0 & 0  \\
    0 & 0 & 0 & 0 & 0 & 0 & * & 0  \\
    0 & 0 & 0 & 0 & 0 & 0 & 0 & *  \\
  \end{array}
\right).
\]}
\end{example}

\section{Type $C$ Pieri Reduction}\label{S:type_c}
In this section, we consider a symplectic  Grassmannian  $X=SG(m,2n)$.
The special Schubert varieties $X_p=X_{\sym_p}$ are indexed by integers $1\leq p\leq 2n-m$.
Since $IG_{\omega}(1,2n) = \P^{2n-1}$, we have
$X_p = \pi(\psi^{-1}(\P(E_{n_p})))$, where   $n_p := 2n+1-m-p$.  Furthermore, we have %The special Schubert symbols are given by
 \begin{equation*}
 \sym_p =
 \begin{cases}
\{n_p\} \cup [2n+2-m,2n] &\text{ if } n_p > m-1, \\
(\{n_p\} \cup [2n+1-m,2n]) \setminus \{2n+1-n_p\} &\text{ if } n_p \leq m-1. \\
 \end{cases}
\end{equation*}

Given Schubert symbols $\B \leq \A\in \mathfrak{S}(X)$, we call $c \in [0,2n]$ a \emph{cut}\footnote{Our definition of \emph{cut} differs from \cite{BKT2}, wherein $c$ must satisfy $\lambda_j \leq c < \mu_{j+1}$ for some $j$.  However subsequent notions are equivalent.} in the Richardson diagram  $D(\A,\B)$ if either
$\lambda_j \leq c < \mu_{j+1}$ or $\lambda_j \leq 2n-c < \mu_{j+1}$ for some $j \in [0,m]$. We notice that   $0$ and $2n$ are always cuts, and set
\begin{align*}
   \mathcal{L}_{\A,\B} &:= \{c \in [1,2n]: \lambda_j < c < \mu_{j+1} \text{ for some } j \in [0,m]\}\\
                       &\,\,\quad \bigcup \{c \in [1,2n]: \mu_j=2n+1-c=\lambda_j \text{ for some } j \in [1,m]\},\\
   \mathcal{Q}_{\A,\B} &:= \{c \in [2,n]: c \text{ is a cut in } D(\A,\B) \text{ and } c-1 \text{ is not}\}.
\end{align*}

 Let   $\{x_j\}$ denote the basis of $V^*$ dual to $\{\+e_j\}$. It is shown in \cite[\S 5]{BKT2} that the projected Richardson variety $Z_{\A,\B}$ is the complete intersection in $\P^{2n-1}$ cut out by the polynomials
\begin{enumerate}
\item $\{x_c: c \in \mathcal{L}_{\A,\B}\}$, and
\item $\{x_{d+1} x_{2n-d} + \ldots + x_c x_{2n+1-c}: c \in \mathcal{Q}_{\A,\B}\}$, where $d$ is the largest cut less than $c$.
\end{enumerate}
We therefore let $$m' := 2n - \#\mathcal{L}_{\A,\B} - \#\mathcal{Q}_{\A,\B} = \dim(Z_{\A,\B})+1\geq 1$$
be the dimension of the affine cone over $Z_{\A,\B}$.
Following \cite[\S 5]{BKT2}, we write $\A \to \B$ if the Richardson diagram $D(\A,\B)$ is defined, contains no $2 \times 2$ blocks of stars, and whenever
it contains two stars in column $c$,   it contains one star in column $2n+1-c$. That is,
\begin{defn}
Given Schubert symbols $\A$ and $\B$ in $\mathfrak{S}(X)$, we write $\A \to \B$ when
\begin{enumerate}
\item $\B \leq \A$,
\item $\lambda_i \leq \mu_{i+1}$ for $1 \leq i \leq m-1$, and
\item if $\lambda_i = \mu_{i+1}$ for some $i$, then $\mu_j < 2n+1-\lambda_i < \lambda_j$ for some $j$.
\end{enumerate}
\end{defn}

Given $\A \to \B$ in $\mathfrak{S}(X)$,  let $\nu(\A, \B) := [1,2n] \setminus \mathcal{L}_{\A,\B}$.
For any subset $\mathcal{I} \subset \mathcal{Q}_{\A,\B}$, we define an associated set
\begin{equation}\label{E:nu_I_type_c}
\sC_{\mathcal{I}}(\A, \B) := \nu(\lambda, \mu) \setminus (\mathcal{I} \cup \{2n+1-c: c \in \mathcal{Q}_{\A,\B}\setminus \mathcal{I}\}),
\end{equation}
which we simply write as $\sC_{\mathcal{I}}$ whenever there is no confusion with $\A, \B$.
Moreover, we shall naturally  think of   $\sC_\mathcal{I}$   as a Schubert symbol
for the type $A$ Grassmannian  $X' := Gr(m', 2n)$, by noting that  the set  $\sC_{\mathcal{I}}$  always has  cardinality
$m'$ for any $\mathcal{I}$.
We therefore obtain projected Richardson varieties $Z_{\sC_\mathcal{I},\sC_{\mathcal{I}}} \subset \mathbb{P}^{2n-1}$ as defined in the case of type $A$ Grassmannians using \eqref{E:projections-new}.  We note that $Z_{\sC_\mathcal{I},\sC_{\mathcal{I}}}$ is the complete intersection in $\mathbb{P}^{2n-1}$ cut out by the linear equations $x_c=0$ for $c\in \mathcal{L}_{\lambda, \mu}$, $x_c=0$ for $c\in \mathcal{I}$, and $x_{2n+1-c}=0$ for $c \in (\mathcal{Q}\setminus \mathcal{I})$.
%We will soon see that the equivariant cohomology class $[Z_{\A, \B}]^T\in H^*_T(\mathbb{P}^{2n-1})$ is in fact given by
%   a sum of the projected Richardson classes $[Z_{\nu_{\mathcal{I}}, \nu_{\mathcal{I}}}]^T$.

We have the following important lemma.
\begin{lemma}\label{L:type_C_degeneration} As classes in $H^*_T(\mathbb{P}^{2n-1})$, we have
\[[Z_{\A,\B}]^T = \sum_{\mathcal{I} \subset \mathcal{Q}_{\A,\B}} [Z_{\sC_\mathcal{I},\sC_{\mathcal{I}}}]^T. \]
\end{lemma}
\begin{proof}
 Let $\zeta := c^T_1(\cO_{\P^{2n-1}}(1)) \in H^*_{T}(\P^{2n-1})$ be the first
equivariant Chern class of the dual to the tautological subbundle
on $\P^{2n-1}$.
For any $j \in [1,2n]$, there is a $T$-invariant hyperplane $Z(x_j)$ defined by the single equation $x_j=0$. We notice that
the natural  $T$-equivariant morphism $\cO_{\P^{2n-1}}(-1) \to \C^{2n} \to \C \+e_j$ defines a $T$-equivariant section of
$\cO_{\P^{2n-1}}(1) \otimes \C \+e_j$, whose zero set is just $Z(x_j)$.
Since
\begin{equation*}
\C\+e_j =
\begin{cases}
\rho^*_X\C_{\chi_j} &\text{ if } 1 \leq j \leq n, \text{ and} \\
\rho^*_X\C_{-\chi_{2n+1-j}}  &\text{ if } n+1 \leq j \leq 2n, \\
\end{cases}
\end{equation*}
the equivariant hyperplane class $[Z(x_j)]^{T}$ is  given by
\begin{equation*}
 [Z(x_j)]^{T} = c^{T}_1(\cO_{\P^{2n-1}}(1) \otimes \C \+e_j) =
 \begin{cases}
\zeta+t_j &\text{ for } 1 \leq j \leq n \\
\zeta-t_{2n+1-j} &\text{ for } n+1 \leq j \leq 2n.
\end{cases}
\end{equation*}
Further details can be found in (for example) \cite[Lecture 4]{fuand}.
For any integers $0 < d < c \leq n$, the quadratic polynomial $f_{d,c} := x_{d+1}x_{2n-d} + \ldots + x_{c}x_{2n+1-c}$
defines a nonzero $T$-equivariant section of the line bundle $\cO_{\P^{2n-1}}(1) \otimes \cO_{\P^{2n-1}}(1)$ over $\P^{2n-1}$.
It follows that the equivariant quadric class $[Z(f_{d,c})]^{T}$ satisfies
\begin{equation*}
 [Z(f_{d,c})]^{T} = c^T_1(\cO_{\P^{2n-1}}(1) \otimes \cO_{\P^{2n-1}}(1)) = 2\zeta.
\end{equation*}
Note that $2\zeta = (\zeta + t_c) + (\zeta - t_c)$ for any $c \in [1,n]$.
In particular, we have
\begin{align*}
 \prod_{c \in \mathcal{Q}_{\A,\B}}(2\zeta)
 &= \prod_{c \in \mathcal{Q}_{\A,\B}}((\zeta+t_c) + (\zeta - t_c))
  = \sum_{\mathcal{I} \subset \mathcal{Q}_{\A,\B}} \left( \prod_{c \in \mathcal{I}}(\zeta+t_c)
 \prod_{c \in \mathcal{Q}_{\A,\B}\setminus \mathcal{I}}(\zeta - t_c)   \right).
\end{align*}
It follows that
\begin{align*}
[Z_{\A,\B}]^{T}
&=\prod_{c \in \mathcal{Q}_{\A,\B}}(2\zeta)
\prod_{c \in \mathcal{L}_{\A,\B}} (\zeta + F(\hat t_c))\\
&=\sum_{\mathcal{I} \subset \mathcal{Q}_{\A,\B}} \prod_{c \in \mathcal{I}}(\zeta+t_c)
\prod_{c \in \mathcal{Q}_{\A,\B}\setminus \mathcal{I}}(\zeta - t_c)
\prod_{c \in \mathcal{L}_{\A,\B}} (\zeta + F(\hat t_c)) \\
&= \sum_{\mathcal{I} \subset \mathcal{Q}_{\A,\B}} [Z_{\sC_{\mathcal{I}},\sC_{\mathcal{I}}}]^{T},\\
\end{align*}
where the last equality is due to the fact that $Z_{\sC_{\mathcal{I}},\sC_{\mathcal{I}}}$
is defined by the linear equations $x_c = 0$ for $c \in [1,2n] \setminus \sC_{\mathcal{I}} =
\mathcal{I} \cup \{2n+1-c: c \in \mathcal{Q}_{\A,\B} \setminus \mathcal{I}\} \cup \mathcal{L}_{\A,\B}$.
\end{proof}

The  maximal torus  $T\subset G=Sp(2n, \C)\subset GL(2n, \C)$ acts on $V=\C^{2n}$ by  diagonal matrices $\mbox{diag}\{z_1, z_2, \cdots, z_n, z_n^{-1}, \cdots, z_2^{-1}, z_1^{-1}\}$.
It is embedded into the maximal torus $\hat T\subset GL(2n, \C)$ of  diagonal matrices $\mbox{diag}\{\hat z_1, \dots, \hat z_{2n}\}$.

This embedding induces natural morphisms
$F: H^*_{\hat T}(\mbox{pt}) = \Z[\hat t_1, \ldots, \hat t_{2n}] \to \Z[t_1, \ldots, t_{n}] = H^*_T(\mbox{pt})$ and
${\bar{F}}: H^*_{\hat T}(\mathbb{P}^{2n-1})\to H^*_T(\mathbb{P}^{2n-1})$,
where $F$ is given by\footnote{In particular, the map $F$ is determined by $c^{\hat T}_1(\hat M)\mapsto c^{T}_1(M)$, where $\hat M$ is any one-dimensional representation of $\hat T$ and $M$ is the restriction of this representation to $T$ via the embedding $T \hookrightarrow {\hat T}$.}
\begin{equation*}
\hat t_i \mapsto
\begin{cases}
t_i &\text{ if } i \leq n, \\
-t_{2n+1-i} &\text{ if } i \geq n+1. \\
\end{cases}
\end{equation*}
Furthermore, the embedding $T \hookrightarrow {\hat T}$ induces the following commutative diagram of morphisms:

 \hspace{3cm}  \xymatrix{
  H^*_{\hat T}(\mathbb{P}^{2n-1}) \ar[d]_{{\int_{\mathbb{P}^{2n-1}}^{\hat T}} } \ar[r]^{{\bar{F}}}
                & H^*_{T}(\mathbb{P}^{2n-1}) \ar[d]^{\int_{\mathbb{P}^{2n-1}}^{ T}\hspace{4cm} \displaystyle(\star)}  \\
 H^*_{\hat T}(\mbox{pt})  \ar[r]^{F}
                & H^*_{T}(\mbox{pt})
                }

The simple roots of $GL(2n,\C)$ are given by $\hat{\alpha}_i=\hat{t}_i-\hat{t}_{i+1}$ for $i=1, \cdots, 2n-1$, and
the simple roots of $G$ are given by $\alpha_n=2t_n$ and  $\alpha_i=t_i-t_{i+1}$ for $i=1, \cdots, n-1$.
Clearly, $F$ sends simple roots of $GL(2n,\C)$ to simple roots of $G$.
%Clearly, $F$ sends positive roots $\hat \alpha_i+\cdots+\hat \alpha_j$ of $GL(2n, \C)$ to positive roots of $G$ with respect to the basis of simple roots.
The Pieri coefficients $N^\B_{\A, p}(X)$ are elements of $\mathbb{Z}_{\geq 0}[-\alpha_1, \cdots, -\alpha_{n}]$, as proven more generally by Graham \cite{Grah}.
We will  express arbitrary Pieri coefficients of type $C_n$ in terms of specializations $F\left(N^{\sC_{\mathcal{I}}}_{\sC_{\mathcal{I}},m+p-m'}\left(Gr(m',2n)\right)\right)$, resulting in a manifestly positive Pieri formula.

Note that any $T$-invariant linear subvariety $L$ of $\P^{2n-1}$ is also invariant under the action of $\hat T$.
By the Borel construction of equivariant cohomology, it follows immediately that ${\bar F}([L]^{\hat T}) = [L]^T$,
where  $[L]^{\hat T}$ denotes the class of $L$ in $H_{\hat T}^*(\P^{2n-1})$.
In particular, since ${\bar F}$ is a ring homomorphism, we have
 \begin{equation}\label{E:borel_construction}
 {\bar F}\left([Z_{\nu_{\mathcal{I},\mathcal{I}}}]^{\hat T} \cdot [\P(E_{n_p})]^{\hat T}\right)
 = [Z_{\nu_{\mathcal{I},\mathcal{I}}}]^{T} \cdot [\P(E_{n_p})]^{T}.
 \end{equation}

\comment{
We will make use of the following well known fact.
\begin{lemma}\label{L:linear_class}
 For any $T$-invariant linear subvariety $L \subset \P^{2n-1}$, we have ${\bar F}([L]^{\hat T}) = [L]^T$. In particular,
 \begin{equation*}
 {\bar F}\left([Z_{\nu_{\mathcal{I},\mathcal{I}}}]^{\hat T} \cdot [\P(E_{n_p})]^{\hat T}\right)
 = [Z_{\nu_{\mathcal{I},\mathcal{I}}}]^{T} \cdot [\P(E_{n_p})]^{T}.
 \end{equation*}
\end{lemma}
\begin{proof}
By the naturality of the equivariant Chern class, we have
\begin{equation*}
{\bar F}([Z(x_j)]^{\hat T}) = {\bar F}\left(c^{\hat T}_1(\cO_{\P^{2n-1}}(1) \otimes \C \+e_j)\right)
=c^T_1(\cO_{\P^{2n-1}}(1) \otimes \C \+e_j)
=[Z(x_j)]^T
\end{equation*}
for any $j \in [1,2n]$.
Any $T$-invariant linear subspace of $\P^{2n-1}$ is a complete intersection of hyperplanes of the form $Z(x_j)$.
The result follows since $\bar F$ is a ring homomorphism.
\end{proof}
}

We now restate and prove Theorem \ref{T:intro},  which reduces arbitrary Pieri coefficients for $SG(m, 2n)$
 to the  specializations of restriction coefficients for $Gr(m',2n)$ via $F$, resulting in a manifestly positive Pieri formula.
\begin{thm}[Equivariant Pieri rule for $SG(m, 2n)$]\label{T:type_c_reduction}
Given $\A, \B\in \mathfrak{S}(X)$ and an integer $1\leq p\leq 2n-m$, we have
$N_{\A, p}^\B=0$ unless  $\A \to \B$  and $|\B| \leq |\A| + p$. When both hypotheses hold,
we define $p':=|\A|+p-|\B|$. We then have
\begin{equation}\label{E:type_c_reduction}
N^\B_{\A,p}(X) = \sum_{\mathcal{I} \subset \mathcal{Q}_{\A,\B}} F\left(N^{\sC_{\mathcal{I}}}_{\sC_{\mathcal{I}},p'}(X')\right).
\end{equation}
%%where we sum over \emph{all} subsets $\mathcal{I} \subset \mathcal{Q}_{\A,\B}$.
Furthermore if  $\B \leq \sym_p$, then
$N^\B_{\A,p}(X) \neq 0$.
\end{thm}

\begin{proof}
Due to Proposition \ref{P:move_to_projective_space}, we can assume $\A \to \B$  and $|\B| \leq |\A| + p$.
By \cite[Propostion 5.1]{BKT2}, we have $\dim(Y_{\A,\B}) + m = \dim(Z_{\A,\B})+1$.  Since
  $\dim(Y_{\A,\B}) = |\B| - |\A|$ and $\dim(Z_{\A,\B})+1 = m'$, we have $p'=m+p-m'$.
Applying
Proposition \ref{P:move_to_projective_space} with $n_p := 2n+1-p-m$,   we  have
\begin{align*}
N^\B_{\A,p}(X)
&=\int^T_{\P^{2n-1}} [Z_{\A,\B}]^{T} \cdot [\P(E_{2n+1-p-m})]^{T} \\
&= \sum_{\mathcal{I} \subset \mathcal{Q}_{\A,\B}} \int^T_{\P^{2n-1}} [Z_{\sC_{\mathcal{I}},\sC_{\mathcal{I}}}]^{T} \cdot [\P(E_{2n+1-p'-m'})]^{T} &\text{ by Lemma \ref{L:type_C_degeneration}} \\
&= \sum_{\mathcal{I} \subset \mathcal{Q}_{\A,\B}} \int^{{T}}_{\P^{2n-1}} {\bar F}([Z_{\sC_{\mathcal{I}},\sC_{\mathcal{I}}}]^{\hat{T}} \cdot [\P(E_{2n+1-p'-m'})]^{\hat{T}}) &\text{ by Equation \ref{E:borel_construction}}\\
&= \sum_{\mathcal{I} \subset \mathcal{Q}_{\A,\B}} F(\int^{\hat{T}}_{\P^{2n-1}} [Z_{\sC_{\mathcal{I}},\sC_{\mathcal{I}}}]^{\hat{T}} \cdot [\P(E_{2n+1-p'-m'})]^{\hat{T}}) &\text{ by diagram } (\star)\\
&= \sum_{\mathcal{I} \subset \mathcal{Q}_{\A,\B}} F({N}^{\sC_{\mathcal{I}}}_{\sC_{\mathcal{I}},p'}(X')).
\end{align*}
For the nonvanishing statement,
note that $\mu_1 = \min(\sC_{\mathcal{I}})$ for any $\mathcal{I} \subset \mathcal{Q}_{\A,\B}$.
As in the type $A$ case, it follows that if $\B \leq \sym_p$, then
$\sC_{\mathcal{I}} \leq \sym_{p'}$,
and hence $F({N}^{\sC_{\mathcal{I}}}_{\sC_{\mathcal{I}},p'}(X')) \neq 0$.
$N^\B_{\A,p}(X)$ is therefore a sum of $2^{\#\mathcal{Q}_{\A,\B}}$ nonzero polynomials
in $\Z_{>0}[t_2-t_1, \ldots, t_n-t_{n-1}, -2t_n]$, and is itself nonzero.
\end{proof}

\begin{example}
Consider $X := SG(3,8)$ with Schubert symbols
$\A = \{2,4,8\}$ and $\B = \{1,3,5\}$, and let $p = 5$.
The Richardson diagram $D(\A,\B)$ is as follows, with
the cuts marked by solid lines:
 \[
\left(
  \begin{array}{cc|cc|cc|cc}
    * & * & 0 & 0 & 0 & 0 & 0 & 0  \\
    0 & 0 & * & * & 0 & 0 & 0 & 0  \\
    0 & 0 & 0 & 0 & * & * & * & *  \\
  \end{array}
\right).
\]
There are no linear equations defining
$Z_{\A,\B}$, and hence the set $\nu(\A, \B) = \{1,2,3,4,5,6,7,8\}$.
Since $\mathcal{Q}_{\A,\B}= \{2,4\}$, we have
the following Schubert symbols for
$X' := Gr(6,8)$:
\begin{align*}
\sC_{\{2,4\}} &= \{1,3,5,6,7,8\},\\
\sC_{\{2\}} &= \{1,3,4,6,7,8\},\\
\sC_{\{4\}} &= \{1,2,3,5,6,8\}, \text{ and} \\
\sC_{\emptyset} &= \{1,2,3,4,6,8\}.
\end{align*}
Thus, we have
\begin{align*}
 N^\B_{\A,5}(X) &= F\left({N}^{\sC_{\{2,4\}}}_{\sC_{\{2,4\}},2}(X') + {N}^{\sC_{\{2\}}}_{\sC_{\{2\}},2}(X')
 + {N}^{\sC_{\{4\}}}_{\sC_{\{4\}},2}(X') + {N}^{\sC_{\emptyset}}_{\sC_{\emptyset},2}(X')\right)\\
 &= F((\hat t_2-\hat t_1)(\hat t_4-\hat t_1)) + F((\hat t_2-\hat t_1)(\hat t_5-\hat t_1)) \\
 &\quad + F((\hat t_4-\hat t_1)(\hat t_7-\hat t_1)) +  F((\hat t_5-\hat t_1)(\hat t_7-\hat t_1))\\
 &= (t_2-t_1)(t_4-t_1) + (t_2-t_1)(-t_4-t_1) \\
 &\quad+ (t_4-t_1)(-t_2-t_1) +  (-t_4-t_1)(-t_2-t_1)\\
 &=4t_1^2.
\end{align*}
We note that in this  example the solution is the square of the highest root.  It will be interesting to find appropriate conditions that simplify the intermediate calculation.
\end{example}

\begin{remark}\label{R:choices}
We mention here that the Pieri formula presented in Theorem \ref{T:type_c_reduction} is just one of several equivalent
formulas that can be derived by our methods.
In particular, let
\[
\mathcal{P} \subset \{c \in [1,n]: \{c,2n+1-c\} \subset \nu(\A,\B)\}
\]
be any subset of cardinality $\#\mathcal{Q}_{\A,\B}$.
Note that $\mathcal{Q}_{\A,\B}$ is itself such a subset.
However, substituting any such $\mathcal{P}$ for $\mathcal{Q}_{\A,\B}$
in \eqref{E:type_c_reduction} and \eqref{E:nu_I_type_c} yields an equivalent formula for $N^{\B}_{\A,p}(X)$.
It is likely that the calculation of certain Pieri coefficients
could be made simpler by using a formula derived from a subset $\mathcal{P} \neq \mathcal{Q}_{\A,\B}$.
We are investigating whether the equivalence of these alternate formulations reveals additional structure of the equivariant cohomology of $X$.
\end{remark}

\section{Type $B$ Pieri Reduction}\label{S:type_b}
Throughout this section, we consider an odd orthogonal Grassmannian  $X=OG(m,2n+1)$.
The special Schubert varieties $X_p$ are indexed by integers $1\leq p\leq 2n-m$, and satisfy
$X_p = \{\Sigma \in X: \dim(\Sigma \cap E_{n_p}) \geq 1\}$, where
\begin{equation*}
n_p :=
\begin{cases}
2n+2-m-p &\text{ if } p \leq n-m,\\
2n+1-m-p &\text{ if } p > n-m.\\
\end{cases}
\end{equation*}
Equivalently, we have $X_p=X_{\sym_p}$, where
\begin{equation*}
\sym_p :=
\begin{cases}
\{n_p\} \cup [2n+3-m,2n+1] &\text{ if } n_p > m-1 \text{, and}\\
(\{n_p\} \cup [2n+2-m,2n+1]) \setminus \{2n+2-n_p\} &\text{ if } n_p \leq m-1.\\
\end{cases}
\end{equation*}

\begin{defn}
Given Schubert symbols $\A$ and $\B$ in $\mathfrak{S}(X)$, we write $\A \to \B$ when
\begin{enumerate}
\item $\B \leq \A$,
\item $\lambda_i \leq \mu_{i+1}$ for $1 \leq i \leq m-1$, and
\item if $\lambda_i = \mu_{i+1}$ for some $i$, then $\mu_j < 2n+2-\lambda_i < \lambda_j$ for some $j$.
\end{enumerate}
\end{defn}

Given Schubert symbols $\B \leq \A$ in $\mathfrak{S}(X)$, we call
$c \in [0,2n+1]$ a \emph{cut} in $D(\A,\B)$ if either
$\lambda_j \leq c < \mu_{j+1}$ or $\lambda_j \leq 2n+1-c < \mu_{j+1}$ for some $j \in [0,m]$.
We notice that $0$ and $2n+1$ are always cuts. Recall that
  $c \in [1,N]$ is a  {zero column} of $D(\A,\B)$ if $\lambda_j < c < \mu_{j+1}$ for some $j$. We set

\begin{align*}
   \mathcal{L}_{\A,\B} &:=\{c \in [1,2n+1]: c \text{ is a zero column in } D(\A,\B)\}\\
			%&:=\{c \in [1,2n+1]: \lambda_j < c < \mu_{j+1} \text{ for some } j \in [0,m]\}\\
                       &\,\,\quad \cup \{c \in [1,2n+1]: \mu_j=2n+2-c=\lambda_j \text{ for some } j \in [1,m]\},\\
   \mathcal{Q}_{\A,\B} &:= \{c \in [2,n+1]: c-1 \text{ is not a cut } \text{, and either } c \text{ is a cut or } c = n+1\}.\\
\end{align*}

Let $x_1,\ldots,x_{2n+1}$ denote the basis of $V^*$ dual to $\+e_1,\ldots,\+e_{2n+1}$.
The projected Richardson variety $Z_{\A,\B}$ is a complete intersection in $\P^{2n-1}$ cut out by the polynomials
\begin{enumerate}
\item $\{x_c: c \in \mathcal{L}_{\A,\B}\}$, and
\item $\{x_{d+1} x_{2n+1-d} + \ldots + x_c x_{2n+2-c}: c \in \mathcal{Q}_{\A,\B}\}$, where $d$ is the largest cut less than $c$.
\end{enumerate}

If $\mathcal{Q}_{\A,\B} \neq \emptyset$, let $\hat{c} \in \mathcal{Q}_{\A,\B}$ be an arbitrary element.
Let $Z'_{\A,\B} \subset \P^{2n}$ denote the subvariety cut out by the same polynomials defining $Z_{\A,\B}$ except the quadratic polynomial corresponding to $\hat{c}$ (namely $x_{\hat{d}+1}x_{2n+1-\hat{d}}+ \ldots + x_{\hat{c}}x_{2n+2-\hat{c}}$, where
$\hat{d}$ is the largest element of $\mathcal{Q}_{\A,\B}$ that is less than $\hat{c}$).
Fixing an integer $1 \leq p \leq 2n-m$, we set the following definitions:
\begin{align*}
\mathcal{Q}'_{\A,\B} &:=
   \begin{cases}
    \mathcal{Q}_{\A,\B} \setminus \{\hat{c}\} &\text{ if } p > n-m \text{ and } \mathcal{Q}_{\A,\B} \neq \emptyset,\\
    \mathcal{Q}_{\A,\B} &\text{ otherwise}.\\
   \end{cases}\\
m' &:= 2n+1 - \#\mathcal{L}_{\A,\B} - \#\mathcal{Q}'_{\A,\B}, \\
X' &:=    Gr(m',2n+1).\\
\end{align*}
Let $\nu(\A,\B) :=  [1,2n+1] \setminus \mathcal{L}_{\A,\B}$.
Let $\nu^+(\A,\B) := \nu(\A,\B) \cup \{n+1\}$.
Finally, for each subset $\mathcal{I} \subset \mathcal{Q'}_{\A,\B}$, we let
\[
\sC_{\mathcal{I}}(\A, \B) := \nu(\A,\B) \setminus \left(\mathcal{I} \cup \{2n+2-c: c \in \mathcal{Q'}_{\A,\B}\setminus \mathcal{I}\}\right).
\]
We will simply denote these sets by $\sC$, $\sC^+$, and $\sC_{\mathcal{I}}$ whenever there is no confusion about $\A,\B$.
Moreover, we will naturally  think of   $\sC_\mathcal{I}$ (resp. $\sC^+$ when  $p > n-m$ and $\mathcal{Q}_{\A,\B} = \emptyset$)  as a Schubert symbol for the type $A$ Grassmannian  $X'$ (resp. $Gr(m'+1,2n+1)$).
We therefore obtain projected Richardson varieties $Z_{\sC_\mathcal{I},\sC_{\mathcal{I}}}$ (resp. $Z_{\sC^+,\sC^+}$)
in $\mathbb{P}^{2n}$, as defined in the case of type $A$ Grassmannians using \eqref{E:projections-new}.

The  maximal   torus $T\subset G=SO(2n+1, \C)$ acts on $V=\C^{2n+1}$ by  diagonal matrices $\mbox{diag}\{z_1, z_2, \cdots, z_n, 1, z_n^{-1}, \cdots, z_2^{-1}, z_1^{-1}\}$.
It is  embedded into a larger torus $(\C^*)^{2n+1}\subset GL(2n+1, \C)$ of diagonal matrices $\mbox{diag}\{\hat z_1, \dots, \hat z_{2n+1}\}$.
This induces a ring homomorphism
$F_B: \Z[\hat t_1, \ldots, \hat t_{2n+1}] \to \Z[t_1, \ldots, t_{n}]$  defined by
\begin{equation*}
\hat{t}_i \mapsto
\begin{cases}
t_i &\text{ if } i \leq n \\
0 &\text{ if } i = n+1 \\
-t_{2n+2-i} &\text{ if } i \geq n+2. \\
\end{cases}
\end{equation*}
 The simple roots of $G$  are given by $\alpha_n=t_n$ and  $\alpha_i=t_i-t_{i+1}$ for $i=1, \cdots, n-1$.
 It is easy to check that the specialization map $F_B$ sends simple roots of $GL(2n+1,\C)$ to simple roots of $G$.
%Clearly, $F_B$ sends a positive root  $\hat \alpha_i+\cdots+\hat \alpha_j$ of $GL(2n+1, \C)$ to (possibly twice of) a positive root of $G$ with respect to the base $\{\alpha_1, \cdots, \alpha_n\}$.
The Pieri coefficients $N^\B_{\A, p}(X)$ are elements of $\mathbb{Z}_{\geq 0}[-\alpha_1, \cdots, -\alpha_{n}]$ by \cite{Grah}.
Using $F_B$, we will  express arbitrary Pieri coefficients for type $B_n$ in terms of specializations of type $A$ coefficients, resulting in a manifestly positive Pieri formula.

We can now state the main result of this section.
\begin{thm}[Equivariant Pieri rule for $OG(m, 2n+1)$]\label{T:type_b_reduction}
Given $\A, \B\in \mathfrak{S}(X)$ and an integer $1\leq p\leq 2n-m$, we have
$N_{\A, p}^\B=0$ unless  $\A \to \B$  and $|\B| \leq |\A| + p$. When both hypotheses hold,
we define $p':= |\A| + p - |\B|$. We then have
\begin{equation*}
N^\B_{\A,p}(X) =
\begin{cases}
\frac{1}{2} F_B\left({N}^{\nu^+}_{\nu^+,p'}(Gr(m'+1,2n+1))\right) &\text{ if } p > n-m  \text{ and } \mathcal{Q}_{\A,\B} = \emptyset, \text{ and} \\
\sum_{\mathcal{I} \subset \mathcal{Q'}} F_B\left({N}^{\nu_{\mathcal{I}}}_{\nu_{\mathcal{I}},p'}(X')\right) &\text{ otherwise}.\\
\end{cases}
\end{equation*}
Furthermore if $\B \leq \sym_p$, then
$N^\B_{\A,p}(X) \neq 0$.
\end{thm}

Consider the $(2n-1)$ dimensional quadric $Q := OG(1,2n+1)$, with inclusion
map $\iota:Q \hookrightarrow \P^{2n}$.
By Proposition \ref{P:move_to_projective_space}, we have
$N^{\B}_{\A,p} = \int^T_{Q} [Z_{\A,\B}]^{T} \cdot [Q \cap \P(E_{n_p})]^{T}$, which we shall reduce
further to a calculation on $H^*_T(\P^{2n})$.
Note that $\P(E_{n_p})$ is not contained in $Q$ if and only if $p \leq n-m$, and in this case
$[Q \cap \P(E_{n_p})]^T = \iota^*[\P(E_{n_p})]^T$.  By the following lemma, we can also express $[Q \cap \P(E_{n_p})]^T = [\P(E_{n_p})]^T$ as the pullback of a class in $H^*_T(\P^{2n})$ when $p > n-m$.

\comment{
In order to prove Theorem \ref{T:type_b_reduction} we must examine the equivariant cohomology of the odd $(2n-1)$ dimensional quadric $Q := OG(1,2n+1)$.
For $1 \leq j \leq 2n-1$, there is exactly one Schubert variety $Q_{j} \subset Q$ of codimension $j$,
relative to the flag $E_{\bull}$.
In particular,
 \begin{equation*}
  Q_{j} =
 \begin{cases}
 \P(E_{2n+1-j}) \cap Q & \text{ if } 1 \leq j \leq n -1 \\
 \P(E_{2n-j}) & \text{ if } n \leq j \leq 2n-1. \\
 \end{cases}
\end{equation*}
The special Schubert variety $X_p$  equals $\pi(\psi^{-1}(Q_{m+p-1}))$ for $1 \leq p \leq 2n-m$.
By Proposition \ref{P:move_to_projective_space}, we therefore have
$N^{\B}_{\A,p} = \int^T_{Q} [Z_{\A,\B}]^{T} \cdot [Q_{m+p-1}]^{T}$, which we shall reduce
further to a calculation on $H^*_T(\P^{2n})$.
}

\begin{lemma}\label{L:special_as_pullback}
Given $n-m < p \leq 2n-m$, we have
\begin{equation*}
[Q \cap \P(E_{n_p})]^T =
   \frac{1}{2} \iota^*[\P(E_{n_p} \oplus \langle \+e_{n+1}\rangle)]^T.
\end{equation*}
\end{lemma}

\begin{proof}
We thank the referee for the following simplified argument.
%If $p \leq n-m$, then $Q_{m+p-1}$ is the complete intersection of $Q$ and $\P(E_{2n+2-m-p})$ in $\P^{2n}$, yielding the desired equality.
As in type $C$, let $\zeta := c^T_1(\cO_{\P^{2n}}(1)) \in H^*_{T}(\P^{2n})$ be the first
equivariant Chern class of the dual to the tautological subbundle
on $\P^{2n}$.
For any $j \in [1,2n+1]$, the corresponding equivariant hyperplane class is defined by
\begin{equation*}
[Z(x_j)]^{T} = \zeta + F_B(\hat{t}_j) =
\begin{cases}
\zeta+t_j &\text{ for } 1 \leq j \leq n \\
\zeta &\text{ for } j = n+1 \\
\zeta-t_{2n+2-j} &\text{ for } n+2 \leq j \leq 2n+1.
\end{cases}
\end{equation*}
For any integers $1 \leq d < c \leq n+1$, the corresponding equivariant quadric class is defined by
\begin{equation*}
 [Z(x_d x_{2n+2-d} + \ldots x_c x_{2n+2-c})]^{T} = 2\zeta.
\end{equation*}
We have
\begin{align*}
&\iota_*\iota^*[\P(E_{n_p} \oplus \langle \+e_{n+1}\rangle)]^T \\
=\hphantom{a} &[\P(E_{n_p} \oplus \langle \+e_{n+1}\rangle)]^T \cdot \iota_*[Q]^T \\
=\hphantom{a} &[\P(E_{n_p} \oplus \langle \+e_{n+1}\rangle)]^T \cdot 2\zeta\\
=\hphantom{a} &2[\P(E_{n_p})]^T \\
=\hphantom{a} &2\iota_*[Q \cap \P(E_{n_p})]^T,
\end{align*}
where the third equality holds because $\zeta = [Z(x_{n+1})]^T$.
Since the pushforward $\iota_*$ is injective, the result follows.
\end{proof}

We can now prove the main result of this section.
\begin{proof}[Proof of Theorem \ref{T:type_b_reduction}]
Due to Proposition \ref{P:move_to_projective_space}, we can assume $\A \to \B$  and $|\B| \leq |\A| + p$.
By \cite[Propostion 5.1]{BKT2}, we have $\dim(Y_{\A,\B}) + m = \dim(Z_{\A,\B})+1$.  Since
$\dim(Y_{\A,\B}) = |\B| - |\A|$ and  $m' = \dim(Z_{\A,\B})+1$ (resp. $\dim(Z_{\A,\B})+2$ if $\mathcal{Q}_{\A,\B} \neq \emptyset \text{ and } p > n-m$), it follows that
$$p'=
\begin{cases}
m+p-m'+1 &\text{ if } \mathcal{Q}_{\A,\B} \neq \emptyset \text{ and } p > n-m,\\
m+p-m' &\text{ otherwise.}\\
\end{cases}$$
We consider the following two cases.

{\bf Case 1}: Suppose $\mathcal{Q}_{\A,\B} = \emptyset$.
In this case $Z_{\A,\B}$ is a linear subvariety of $Q \subset \P^{2n}$ and is equal to $Z_{\nu,\nu}$.
Therefore, by Proposition \ref{P:move_to_projective_space}, we have
$N^{\B}_{\A,p} = \int^T_{Q}[Z_{\nu,\nu}]^{T} \cdot [Q \cap \P(E_{n_p})]^T.$

If $p \leq n-m$, then
$$[Q \cap \P(E_{n_p})]^T= \iota^*[\P(E_{2n+2-m-p})]^T = \iota^*[\P(E_{2n+2-m'-p'})]^T.$$
By the projection formula we conclude $N^{\B}_{\A,p} = F_B\left(N^\nu_{\nu,p'}(X')\right)$.

Now, observe that we must have $n+1 \in \mathcal{L}_{\A,\B}$
(since $n$ and $n+1$ are both cuts and $n+1 \not\in \A \cup \B$).
Recall that $Z(x_{n+1}) \subset \P^{2n}$ is the hyperplane cut out by $x_{n+1} = 0$.
If $p > n-m$, then by Lemma \ref{L:special_as_pullback} we have
\begin{align*}
N^{\B}_{\A,p} &= \frac{1}{2} \int^T_{\P^{2n}}[Z_{\nu,\nu}]^{T} \cdot [\P(E_{n_p} \oplus \langle \+e_{n+1}\rangle)]^T \\
&= \frac{1}{2} \int^T_{\P^{2n}}[Z_{\nu^+,\nu^+}]^{T} \cdot [Z(x_{n+1})]^T \cdot [\P(E_{n_p} \oplus \langle \+e_{n+1}\rangle)]^T \\
%&= \frac{1}{2} \rho_{*}\left(\prod_{c \in \mathcal{L}_{\A,\B}} (\zeta + F(t_c)) \prod_{\substack{c \geq 2n+2-p-m \\ c \neq n+1}} (\zeta + F(t_c))\right)\\
%&= \frac{1}{2} \rho_{*}\left(\zeta \cdot \prod_{c \in \mathcal{L}_{\A,\B} \setminus \{n+1\}} (\zeta + F(t_c)) \prod_{\substack{c \geq 2n+2-p-m \\ c \neq n+1}} (\zeta + F(t_c))\right)\\
%&= \frac{1}{2} \rho_{*}\left(\prod_{c \in \mathcal{L}_{\A,\B} \setminus \{n+1\}} (\zeta + F(t_c)) \prod_{c \geq 2n+2-p-m} (\zeta + F(t_c))\right)\\
&= \frac{1}{2} \int^T_{\P^{2n}}[Z_{\nu^+,\nu^+}]^{T} \cdot  [\P(E_{n_p})]^T \\
&= \frac{1}{2} \int^T_{\P^{2n}}[Z_{\nu^+,\nu^+}]^{T} \cdot  [\P(E_{2n+2-(m'+1)-p')})]^T \\
&= \frac{1}{2} F_B\left({N}^{\nu^+}_{\nu^+,p'}(X')\right).
\end{align*}

{\bf Case 2}: Now suppose $\mathcal{Q}_{\A,\B} \neq \emptyset$.
Note that $Z'_{\A,\B} \cap Q = Z_{\A,\B}$ and $\iota^*[Z'_{\A,\B}]^T = [Z_{\A,\B}]^T \in H^*_T(Q)$.
As in Case 1, we could use Lemma \ref{L:special_as_pullback} along with the projection formula to
move our calculation to $\P^{2n}$ and prove the desired result.
However, we will instead use the fact
that $[Z_{\A,\B}]^T = \iota^*[Z'_{\A,\B}]^T$ along with the projection formula to move to $\P^{2n}$.
The advantage of this alternate proof is that it generalizes immediately to the type $D$ case,
where there is no analogue of Lemma \ref{L:special_as_pullback}.

Since $\P(E_{n_p}) \subset Q$ for $p > n-m$, we have
\begin{equation*}
\iota_*[\P(E_{n_p}) \cap Q]^{T} =
\begin{cases}
[\P(E_{n_p})]^T \cdot [Q]^T  &\text{ if } p \leq n-m\\
[\P(E_{n_p})]^T  &\text{ if } p > n-m\\
\end{cases}
\end{equation*}
We therefore have
\begin{align*}
N^\B_{\A,p}(X) &= \int^T_{Q} [Z_{\A,\B}]^{T} \cdot [\P(E_{n_p}) \cap Q]^{T} \\
&=
\begin{cases}
\int^T_{\P^{2n}} [Z'_{\A,\B}]^{T} \cdot [\P(E_{n_p})]^T \cdot [Q]^T &\text{ if } p \leq n-m\\
\int^T_{\P^{2n}} [Z'_{\A,\B}]^{T} \cdot [\P(E_{n_p})]^T &\text{ if } p > n-m\\
\end{cases}\\
&=
\begin{cases}
\int^T_{\P^{2n}} [Z_{\A,\B}]^{T} \cdot [\P(E_{n_p})]^T &\text{ if } p \leq n-m\\
\int^T_{\P^{2n}} [Z'_{\A,\B}]^{T} \cdot [\P(E_{n_p})]^T &\text{ if } p > n-m\\
\end{cases}\\
&= \int^T_{\P^{2n}}\sum_{\mathcal{I} \subset \mathcal{Q'}_{\A,\B}} [Z_{\nu_\mathcal{I},\nu_\mathcal{I}}]^T \cdot [\P(E_{2n+2-m'-p'})]^T \\
&= \sum_{\mathcal{I} \subset \mathcal{Q'}_{\A,\B}} F_B\left({N}^{\sC_{\mathcal{I}}}_{\sC_{\mathcal{I}},p'}(X')\right),
\end{align*}
where the subvarieties $Z_{\nu_\mathcal{I},\nu_\mathcal{I}}$ are defined as in \eqref{E:projections-new}.
The fact that $[Z_{\A,\B}]^{T}$ (resp. $[Z'_{\A,\B}]^{T}$ for $p > n-m$) is equal to $\sum_{\mathcal{I} \subset \mathcal{Q'}_{\A,\B}} [Z_{\nu_\mathcal{I},\nu_\mathcal{I}}]^T$ follows from Lemma \ref{L:type_C_degeneration}.
For the nonvanishing statement, let $\sym$ be any of the Schubert symbols $\nu_{\mathcal{I}}$ in Cases $1$ and $3$ (resp. $\nu^+$ in Case $2$),
and note that $\mu_1 = \min(\sym)$.
As in the type $A$ case, it follows that if $\B \leq \sym_p$, then
$\sym \leq \sym_{p'}$,
and hence $F_B({N}^{\sym}_{\sym,p'}(X')) \neq 0$.
In Cases $1$ and $3$, it follows that $N^\B_{\A,p}(X)$ is a sum of $2^{\#\mathcal{Q}'_{\A,\B}}$ polynomials
in $\Z_{>0}[t_2-t_1, \ldots, t_n-t_{n-1}, -t_n]$, and in particular is nonzero.
\end{proof}

\begin{remark}\label{R:type_b_restrictions}
When $p > n-m$ and $\mathcal{Q}_{\A,\B} = \emptyset$, the
Pieri coefficient $N^{\mu}_{\la,p}(X) = \frac{1}{2} F_B({N}^{\nu^+}_{\nu^+,p'}(X'))$
is indeed an element of $\Z[t_1,\ldots,t_n]$,
since $2$ divides $\iota^*[\P(E_{n_p} \oplus \langle \+e_{n+1}\rangle)]^T$
by Lemma \ref{L:special_as_pullback}.
By Lemma \ref{L:restriction_coef_of_same_Lie_type},
we will see that $N^{\mu}_{\la,p}(X)$ is equal to the
type $B$ restriction coefficient $N^{\nu}_{\nu,p'}(OG(m',2n+1))$.
In particular, any restriction coefficient of Lie type $B$
is the specialization of a restriction coefficient of Lie type $A$, up to a factor of $2$.
\end{remark}

\begin{example}
 Let $X := OG(2,7)$, let $p = 3$, and consider Schubert symbols $\A = \{3,6\}$ and $\B = \{1,6\}$.
 Note that $p > n-m = 1$.
 We have $\mathcal{L}_{\A,\B} = \{2,4,5,7\}$, $\mathcal{Q}_{\A,\B} = \emptyset$, $m' = 3$, $p' = 2$, $X' = Gr(4,7)$,
 and $\nu^+ = \{1,3,4,6\}$.  Using Lemma \ref{L:restriction_formula2}, we calculate
 $N^{\nu^+}_{\nu^+,2}(X') = (\hat{t}_5-\hat{t}_1)(\hat{t}_7-\hat{t}_1)$.  Hence $N^\B_{\A,3} = \frac{1}{2}(-t_3-t_1)(-2t_1) = (-t_3-t_1)(-t_1)$.
\end{example}

\section{Type $D$ Pieri Reduction}\label{S:type_d}
Throughout this section, we consider an even orthogonal  Grassmannian  $X=OG(m,2n)$, where $m \leq n$.
We note that our results apply equally well to $OG(n,2n)$, which has two connected components, and to $OG(n-1,2n)$, which is realized as a homogeneous space by quotienting $SO(2n,\C)$ by a \emph{submaximal} parabolic subgroup.
 %and let $\tilde{E}_{\bull}$
 %%be the complete flag on $\C^{2n}$ defined by replacing $E_n$ by $\tilde{E}_{n}$ in $E_{\bull}$.
 For each $1 \leq p \leq 2n-m-1$, the special Schubert variety $X_p$ has codimension $p$ and satisfies
$X_p = \{\Sigma \in X: \dim(\Sigma \cap E_{n_p}) \geq 1\}$, where
\begin{equation*}
n_p :=
\begin{cases}
2n+1-m-p &\text{ if } p < n-m \\
2n - m - p &\text{ if } p \geq n-m.\\
\end{cases}
\end{equation*}
Equivalently, we have $X_p = X_{\sym_p}$, where
 \begin{equation*}
 \sym_p :=
 \begin{cases}
\{n_p\} \cup [2n+2-m,2n] &\text{ if } n_p > m-1, \\
(\{n_p\} \cup [2n+1-m,2n]) \setminus \{2n+2-n_p\} &\text{ if } n_p \leq m-1, \\
 \end{cases}
\end{equation*}
We mention here that there is an additional special Schubert variety
$\tilde{X}_{n-m}$ of codimension $n-m$ in $X$.  We will discuss multiplication by the Schubert class $[\tilde{X}_{n-m}]^T$ in Remark \ref{R:additional_schubert_class_type_d}.

Following \cite[\S A]{BKT4} we make the following definitions.
For any Schubert symbol $\A \in \mathfrak{S}(X)$, let
$[\A] = \A \cup  \{c \in [1,2n]: 2n+1-c \in \A\}$.
We define
$\type(\A) \in \{0,1,2\}$ as follows.
If $n \in [\A]$,
then we let $\type(\A)$
be congruent mod $2$ to $1$ plus the
number of
elements in $[1,n] \setminus \A$.  In other words,
if $\#([1,n] \setminus \A)$ is even
then $\type(\A)=1$, and
if $\#([1,n] \setminus \A)$ is odd
then $\type(\A)=2$.
Finally, if  $n \not\in [\A]$,
we set $\type(\A)=0$.

In type $D$, the Bruhat order is characterized by a stricter relation on the Schubert symbols.
\begin{defn}
Given Schubert symbols $\A$ and $\B$ in $\mathfrak{S}(OG(m,2n))$,
we write $\B \preceq \A$ if the following conditions hold:
\begin{enumerate}
\item $\B \leq \A$, and
\item if there exists $c \in [1,n-1]$ such that $[c+1,n] \subset [\A] \cap [\B]$
and $\#\A \cap [1,c] = \#\B \cap [1,c]$, then
we have $\type(\A) = \type(\B)$.
\end{enumerate}
\end{defn}
By  \cite[Proposition A.2]{BKT4}, we have $\B \preceq \A$ if and only if $X_\B \subset X_\A$.
\begin{defn}
Given Schubert symbols $\A$ and $\B$ in $\mathfrak{S}(X)$, we write $\A \to \B$ when
\begin{enumerate}
\item $\B \preceq \A$;
\item $\lambda_i \leq \mu_{i+1}$ for $1 \leq i \leq m-1$, unless $\lambda_i = n+1$ and $\mu_{i+1} = n$;
\item if $\lambda_i = \mu_{i+1}$ for some $i$, then $\mu_j < 2n+2-\lambda_i < \lambda_j$ for some $j$; and
\item it is not the case that $\lambda_i = \mu_{i+1} \in \{n,n+1\}$ for any $1 \leq i \leq m-1$.
\end{enumerate}
\end{defn}

\begin{remark}
Given arbitrary Schubert symbols $\A, \B \in \mathfrak{S}(X)$,
the definition of a cut in $D(\A,\B)$ is significantly more complicated in type $D$ (see \cite{Ravi} for details).
However, assuming $\A \to \B$, the situation is simpler.
\end{remark}

Given Schubert symbols $\A \to \B \in \mathfrak{S}(X)$, we call
$c \in [0,2n]$ a \emph{cut} in $D(\A,\B)$ if either
$\lambda_j \leq c < \mu_{j+1}$ or $\lambda_j \leq 2n-c < \mu_{j+1}$ for some $j \in [0,m]$.
We notice that $0$ and $2n$ are always cuts. Recall that
  $c \in [1,2n]$ is a  {zero column} of $D(\A,\B)$ if $\lambda_j < c < \mu_{j+1}$ for some $j$.
The following definitions also hold only for Schubert symbols $\A \to \B$:
\begin{align*}
   \mathcal{L}_{\A,\B} :=& \{c \in [1,2n]: c \text{ is a zero column in } D(\A,\B)\}\\
			%&:=\{c \in [1,2n+1]: \lambda_j < c < \mu_{j+1} \text{ for some } j \in [0,m]\}\\
                       &\cup \{c \in [1,2n]:
                       \begin{aligned}[t]
                       &\exists j \in [1,m] \text{ such that } 2n+1-c=\mu_j=\lambda_j \\
                       &\text{or }2n+1-c=n+1=\lambda_j<\mu_{j+1} \\
                       &\text{or } 2n+1-c=n=\mu_j>\lambda_{j-1}\},
                       \end{aligned}\\
   \mathcal{Q}_{\A,\B} := &\{c \in [2,n]: c-1 \text{ is not a cut } \text{, and either } c \text{ is a cut or } c = n\}.\\
\end{align*}
Let $x_1,\ldots,x_{2n}$ denote the basis of $V^*$ dual to $\+e_1,\ldots,\+e_{2n}$.
The projected Richardson variety $Z_{\A,\B}$ is a complete intersection in $\P^{2n-1}$ cut out by the polynomials
\begin{enumerate}
\item $\{x_c: c \in \mathcal{L}_{\A,\B}\}$, and
\item $\{x_{d+1} x_{2n-d} + \ldots + x_c x_{2n+1-c}: c \in \mathcal{Q}_{\A,\B}\}$, where $d$ is the largest cut less than $c$.
\end{enumerate}

If $\mathcal{Q}_{\A,\B} \neq \emptyset$, let $\hat{c} \in \mathcal{Q}_{\A,\B}$ be an arbitrary element.
Let $Z'_{\A,\B} \subset \P^{2n-1}$ denote the subvariety cut out by the same polynomials defining $Z_{\A,\B}$ except the quadratic polynomial corresponding to $\hat{c}$.
Fixing an integer $1 \leq p \leq 2n-m-1$, we set the following definitions:
\begin{align*}
\mathcal{Q}'_{\A,\B} &:=
\begin{cases}
\mathcal{Q}_{\A,\B} \setminus \{\hat{c}\} &\text{ if } \mathcal{Q}_{\A,\B} \neq \emptyset \text{ and } p \geq n-m,\\
\mathcal{Q}_{\A,\B} &\text{ otherwise}.\\
\end{cases}\\
m' &:= 2n - \#\mathcal{L}_{\A,\B} - \#\mathcal{Q}'_{\A,\B}, \\
X' &:= Gr(m',2n).\\
\end{align*}
Let $\nu(\A,\B) :=  [1,2n] \setminus \mathcal{L}_{\A,\B}$.
For each subset $\mathcal{I} \subset \mathcal{Q'}_{\A,\B}$, we let
\[
\sC_{\mathcal{I}}(\A, \B) := \nu(\A,\B) \setminus \left(\mathcal{I} \cup \{2n+1-c: c \in \mathcal{Q'}_{\A,\B}\setminus \mathcal{I}\}\right).
\]
We will simply denote these sets as $\sC$ and $\sC_{\mathcal{I}}$ whenever there is no confusion about $\A,\B$.
Moreover, we shall naturally  think of   $\sC_\mathcal{I}$   as a Schubert symbol
for the type $A$ Grassmannian  $X' = Gr(m',2n)$, by noting that  the set  $\sC_{\mathcal{I}}$  always has  cardinality
$m'$ for any $\mathcal{I}$.
We therefore obtain projected Richardson varieties $Z_{\sC_\mathcal{I},\sC_{\mathcal{I}}}\subset \mathbb{P}^{2n-1}$ as defined in the case of type $A$ Grassmannians using \eqref{E:projections-new}.

The  maximal   torus $T\subset G=SO(2n, \C)$ acts on $V=\C^{2n}$ by  diagonal matrices $\mbox{diag}\{z_1, z_2, \cdots, z_n, z_n^{-1}, \cdots, z_2^{-1}, z_1^{-1}\}$.
It is embedded in a larger torus $(\C^*)^{2n}\subset GL(2n, \C)$ of diagonal matrices $\mbox{diag}\{\hat z_1, \dots, \hat z_{2n}\}$.
This induces a ring homomorphism
$F_D: \Z[\hat t_1, \ldots, \hat t_{2n}] \to \Z[t_1, \ldots, t_{n}]$  defined by
\begin{equation*}
\hat{t}_i \mapsto
\begin{cases}
t_i &\text{ if } i \leq n \\
-t_{2n+1-i} &\text{ if } i \geq n+1. \\
\end{cases}
\end{equation*}
The simple roots of $G$  are given by $\alpha_n=t_{n-1}+t_n$ and  $\alpha_i=t_i-t_{i+1}$ for $i=1, \cdots, n-1$.
The map   $F_D$  sends positive roots of $GL(2n, \C)$ to positive roots of $G$, with the sole exception of the root $\hat{\alpha}_{n} = \hat{t}_n - \hat{t}_{n+1}$, which is sent to $2t_n$.
However, using Corollary \ref{C:bad_terms_do_not_occur}, we will ensure that
specializations of the root $\hat{\alpha}_{n}$ never occur in our Pieri rule.
We can now state the main result of this section.
\begin{thm}[Equivariant Pieri rule for $OG(m, 2n)$]\label{T:type_d_reduction}
Given $\A, \B\in \mathfrak{S}(X)$ and an integer $1\leq p\leq 2n-m-1$, we have
$N_{\A, p}^\B=0$ unless  $\A \to \B$  and $|\B| \leq |\A| + p$. When both hypotheses hold,
we define  $p':=|\A|+p-|\B| \geq 0$. We then have
\begin{equation}\label{E:type_d_pieri_part1}
N^\B_{\A,p}(X) =
\begin{cases}
N^{\nu}_{\nu,p'}(OG(m',2n)) &\text{ if } \mathcal{Q}_{\A,\B} = \emptyset \text{ and } p \geq n-m, \\
\sum_{\mathcal{I} \subset \mathcal{Q}'_{\A,\B}} F_D\left({N}^{\nu_{\mathcal{I}}}_{\nu_{\mathcal{I}},p'}(X')\right) &\text{ otherwise}.\\
\end{cases}
\end{equation}
Furthermore if  $\B \preceq \sym_p$, then
$N^\B_{\A,p}(X) \neq 0$.
\end{thm}

The following lemma holds for all isotropic Grassmannians.  We will
need it to prove Theorem \ref{T:type_d_reduction} in the case $\mathcal{Q}_{\A,\B} = \emptyset$ and $p \geq n-m$.
\begin{lemma}\label{L:restriction_coef_of_same_Lie_type}
Let $IG_{\omega}(m,N)$ be a Grassmannian of Lie type $D_n$ (respectively $B_n$ or $C_n$), and
suppose $\A\to \B$ and $1 \leq p \leq 2n-m-1$ (respectively $1 \leq p \leq 2n-m$).
If $\mathcal{Q}_{\A,\B} = \emptyset$, then we  have $\#\nu\leq n$,   where $\nu=\nu(\A, \B)= [1,N] \setminus \mathcal{L}_{\A,\B}$.
Furthermore if   $|\B| \leq |\A| + p$, then
$$N^{\B}_{\A,p}(IG_{\omega}(m,N)) = N^{\nu}_{\nu, |\A|+p-|\B|}(IG_{\omega}(\#\nu,N)).$$
\end{lemma}
\begin{proof}
If $\mathcal{Q}_{\A,\B} = \emptyset$, then $c$ is a cut in $D(\A,\B)$ for
any $c \in [1,N]$.
We then claim that $\{c, N+1-c\} \cap \mathcal{L}_{\A,\B} \neq \emptyset$
for any $c \in [1,N]$.
To see why, we reproduce the argument from \cite[Lemma 4.14]{Ravi}:
If we are working in type $B$ and $c = n + 1$, then $n + 1$ must be a zero
column, and hence be in $\mathcal{L}_{\A,\B}$ . Otherwise, let us consider $c \leq N/2$.
Since $c-1$ is a cut, we have $\lambda_i \leq c - 1 < \mu_{i+1}$  (or $\lambda_i \leq N + 1 - c < \mu_{i+1}$) for some $i$. Then  $c$ (resp. $N + 1 - c$) is a zero column, in which case we are done; or we have $c = \mu_{i+1} \in \B$  (resp. $N + 1 - c = \lambda_i \in \A$).
It follows that $N+1-c \in \mathcal{L}_{\A,\B}$ (resp. $c \in \mathcal{L}_{\A,\B}$), proving our claim.
We therefore have $\#\nu=N-\#\mathcal{L}_{\A,\B} \leq n$.

As a consequence,   $\nu$  can be treated as   a Schubert symbol for $IG_{\omega}(\#\nu,N)$.
We now  consider $Z_{\nu,\nu} \subset \P^{N-1}$ to be the projected Richardson variety
coming from the Richardson variety $Y'_{\nu,\nu} \subset IG_{\omega}(\#\nu,N)$.
As in Lemma \ref{L:same_zero_columns}, we then see that $Z_{\nu,\nu} = Z_{\A,\B}$.
Since $|\A|+p-|\B|\geq 0$, we have
$N^{\B}_{\A,p}(IG_{\omega}(m,N)) = N^{\nu}_{\nu,|\A|+p-|\B|}(IG_{\omega}(\#\nu,N))$ by Proposition \ref{P:move_to_projective_space}.
\end{proof}

The proof of our main theorem follows easily.
\begin{proof}[Proof of Theorem \ref{T:type_d_reduction}]
We note that as in the proof of Theorem \ref{T:type_b_reduction}, we have
$p = m+p-m'+1$ if $\mathcal{Q}_{\A,\B} \neq \emptyset$ and $p \geq n-m$, and
$p' = m+p-m'$ otherwise.
If $\mathcal{Q}_{\A,\B}=\emptyset$ and $p \geq n-m$ then we are done by Lemma \ref{L:restriction_coef_of_same_Lie_type}.
Let $Q := OG(1,2n)$ denote the $(2n-2)$ dimensional quadric of isotropic lines, with inclusion $\iota:Q \hookrightarrow \P^{2n-1}$.
Note that $[Q \cap \P(E_{n_p})]^T = \iota^*[\P(E_{n_p})]^T$ if and only if $p < n-m$.
If $\mathcal{Q}_{\A,\B}=\emptyset$ and $p < n-m$, it follows that
$N^{\B}_{\A,p} = \int^T_{Q}[Z_{\nu,\nu}]^{T} \cdot \iota^*[\P(E_{n_p})]^T
=F_D\left(N^\nu_{\nu,p'}(X')\right)$,
as in the proof of Case 1 of Theorem \ref{T:type_b_reduction}.
Finally, if $\mathcal{Q}_{\A,\B} \neq \emptyset$, the proof is identical to
the proof of Case 2 of Theorem \ref{T:type_b_reduction} (simply replace $F_B$ with $F_D$, $2n$ with $2n-1$, $p \leq n-m$ with $p < n-m$, and $p > n-m$ with $p \geq n-m$).
\end{proof}

In conclusion, the combination of Theorem \ref{T:type_d_reduction} with Lemma \ref{L:restriction_formula2} and Corollary \ref{C:bad_terms_do_not_occur} yields a manifestly positive type $D$ Pieri rule.
To see that the hypotheses of Corollary \ref{C:bad_terms_do_not_occur} are met, note that
if $I_1 \cap \nu_\mathcal{I}=\{n\}$, then $\mu_1 = n$, since $\mu_1 \in \nu_\mathcal{I}$ for any $\mathcal{I} \subset \mathcal{Q}_{\lambda,\mu}$.  Therefore $c$ is a cut for every $c \in [1,n]$, and hence $\mathcal{Q}_{\lambda,\mu} = \emptyset$.
Moreover, since $I_1 = [1,n]$, we have $2n+1-p'-m' = n$, so $p = n-m+1 > n-m$.
In this case, the Pieri rule does not make use of the specialization $F_D$, and positivity
follows from any one of the known positive formulas for the restriction coefficient $N^{\nu}_{\nu,p'}(OG(m',2n))$ (e.g. \cite{koku,AJS,Billey}).
%\begin{example}
% Let $X := OG(2,8)$, $\A :=\{2,8\}$, $\B :=\{1,7\}$, and suppose we wish to calculate
%$N^{\B}_{\A,2}(X)$.  Note that $\nu(\la,\mu) = \{1,2,7,8\}$,
% $\#\mathcal{Q}_{\A,\B} = 1$,  $\#\mathcal{Q}'_{\A,\B} = 0$, $m' = 4$, $p'=1$, and $X' = Gr(4,8)$.
% It follows that $N^{\B}_{\A,2}(X) = F_D(N^{\nu}_{\nu,1}(X')) =
% F_D\left(\hat t_6 - \hat t_2)+(\hat t_5 - \hat t_1)\right)
% = (-t_3 - t_2)+(-t_4 - t_1).$
%\end{example}
\begin{example}\label{E:type_d_example}
 Let $X := OG(1,8)$,and suppose we wish to calculate
 $N^{\{1\}}_{\{2\},4}(X)$.  Note that $\#\mathcal{Q}_{\{2\},\{1\}} = \emptyset$,  and $p=4>3=n-m$.
 It follows that
 \begin{align*}
 N^{\{1\}}_{\{2\},4}(X) &= N^{\{1,2\}}_{\{1,2\},3}(OG(2,8)) \\
 &=(-t_1-t_2)\left((-t_4 - t_2)(t_4-t_2) + (-t_2-t_1)(-t_3-t_1)\right),
 \end{align*}
 where the final polynomial is computed using Anders Buch's Equivariant Schubert Calculator \cite{Buch_calculator}.  Interestingly,
 this type $D$ restriction coefficient is not the specialization of any type $A$ restriction coefficient
 (see Remark \ref{R:type_d_restrictions}).
\end{example}
\begin{remark}\label{R:additional_schubert_class_type_d}
  For the even orthogonal Grassmannian $OG(m, 2n)$, there is an additional special Schubert variety $\tilde{X}_{n-m} $ of codimension $n-m$ defined by
$\tilde{X}_{n-m} := \{\Sigma \in X: \dim(\Sigma \cap \langle \+e_1, \ldots, \+e_{n-1}, \+e_{n+1}\rangle) \geq 1\}$ \footnote{The definitions of $\tilde{X}_{n-m}$ and ${X}_{n-m}$ are reversed in \cite{BKT2} when $n$ is even.}. The type $D_n$ Dynkin diagram automorphism
induces an involution of $OG(m, 2n)$ that interchanges Schubert varieties of types $1$ and $2$, and maps $\tilde{X}_{n-m} $ to ${X}_{n-m}$. Multiplication  with $[\tilde{X}_{n-m}]^T$  is equivalent to   multiplication with $[{X}_{n-m}]^T$ in the following sense:  We denote by $\tilde{N}^\B_{\A,n-m}(X)$ the structure coefficients of   $[X_\lambda]^T\cdot [\tilde{X}_{n-m}]^T$. Then, it follows from \eqref{E:type_d_pieri_part1} and the involution of $OG(m, 2n)$ that
\begin{equation*}\label{E:type_d_pieri_part2}
\tilde{N}^\B_{\A,p}(X) =
\begin{cases}
\tilde{N}^{\nu}_{\nu,p'}(X') &\text{ if } \mathcal{Q}_{\A,\B} = \emptyset, \text{ and} \\
\sum_{\mathcal{I} \subset \mathcal{Q'}_{\tilde{\A},\tilde{\B}}} \tilde{F}_D\left({N}^{\tilde{\nu}_{\mathcal{I}}}_{\tilde{\nu}_{\mathcal{I}},p'}(X')\right) &\text{ otherwise};\\
\end{cases}
\end{equation*}
where $X'$ and $p'$ are defined as before, and
\begin{equation*}
\tilde{F}_D(\hat t_{i}):=
\begin{cases}
 {F}_D(\hat t_{2n+1-i}) &\text{ if } i \in \{n,n+1\},\text{ and} \\
 {F}_D(\hat t_{i}) &\text{ if } i\notin\{n, n+1\} \\
\end{cases}
\end{equation*}
is induced by $F_D$ and the involution.
We note that $\tilde F_D$ also sends all positive roots of type $A_{2n-1}$ to positive roots of type $D_n$ except for the simple root
  $\hat t_n-\hat t_{n+1}$. It follows from Corollary \ref{C:bad_terms_do_not_occur} that  the aforementioned formula for $\tilde{N}^\B_{\A,p}(X)$ is manifestly positive.
\end{remark}
\begin{remark}\label{R:type_d_restrictions}
Suppose $\mathcal{Q}_{\A,\B} = \emptyset$ and
$p \geq n-m$. The restriction coefficient
$N^\B_{\A,p} = N^{\nu}_{\nu,p'}(OG(m',2n))$ may not be the specialization (via $F_D$ or $\tilde{F}_D$)
of a type $A$ restriction coefficient from a Grassmannian $Gr(m'',2n)$ for any $1 \leq m'' \leq 2n-1$, even up to a factor of $2$ (in contrast to the type $B$ case; see Remark \ref{R:type_b_restrictions}).
We have verified several such examples by computer, including Example \ref{E:type_d_example}.
\end{remark}

\iffalse

In order to prove equation \eqref{E:type_d_pieri_part2}, let $G:\Z[t_1,\ldots,t_n] \to \Z[t_1,\ldots,t_n]$ be the ring homomorphism defined by $t_j \mapsto t_j$ for $1 \leq j \leq n-1$ and $t_n \mapsto -t_n$.
Since $\#\mathcal{Q}_{\A,\B} = \#\mathcal{Q}_{\tilde{\A},\tilde{\B}}$,
$\mathcal{L}_{\A,\B} = G(\mathcal{L}_{\tilde{\A},\tilde{\B}})$, and
$\{c: \tilde{E}_{n} \text{ satisfies } x_c = 0\} = G(\{c: E_{n} \text{ satisfies } x_c = 0\})$, it follows that
$\int^T_{\P^{2n-1}} [Z'_{\A,\B}]^{T} \cdot [\P(\tilde{E}_{n})]^{T} = G(\int^T_{\P^{2n-1}} [Z'_{\tilde{\A},\tilde{\B}}]^{T} \cdot [\P(E_n)]^T)$.
By Proposition \ref{P:move_to_projective_space}, we have
\begin{align*}
\tilde{N}^\B_{\A,n-m}(X) &= \int^T_{\P^{2n-1}} [Z'_{\A,\B}]^{T} \cdot [\P(\tilde{E}_{n})]^{T} \\
&= G\left(\int^T_{\P^{2n-1}} [Z'_{\tilde{\A},\tilde{\B}}]^{T} \cdot [\P(E_n)]^T\right) \\
&= G\left(\int^T_{\P^{2n-1}} \sum_{\mathcal{I} \subset \mathcal{Q'}_{\tilde{\A},\tilde{\B}}} [Z_{\tilde{\nu}_\mathcal{I},\tilde{\nu}_\mathcal{I}}]^T \cdot [\P(E_n)]^T\right) \\
&= G\left(\sum_{\mathcal{I} \subset \mathcal{Q'}_{\tilde{\A},\tilde{\B}}} F_D\left({N}^{\tilde{\sC}_{\mathcal{I}}}_{\tilde{\sC}_{\mathcal{I}},n-m}(X')\right)\right) \\
&= \sum_{\mathcal{I} \subset \mathcal{Q'}_{\tilde{\A},\tilde{\B}}} \tilde{F}_D\left({N}^{\tilde{\sC}_{\mathcal{I}}}_{\tilde{\sC}_{\mathcal{I}},n-m}(X')\right).\\
\end{align*}

\fi
\appendix

\section{Formula for Restrictions of Special Schubert Classes}\label{S:appendix}
In this appendix, we provide a  manifestly positive formula for the restriction of a special Schubert class in $H_T^*(Gr(m, N))$ to an arbitrary $T$-fixed point, stated in terms of Schubert symbols.  We derive this formula directly from the Atiyah-Bott-Berline-Verge integration formula, but it can also be deduced from other restriction formulas, such as \cite{buch_rimanyi, Billey}.
We thank Sushmita Venugopalan for her insight in proving the following key lemma.
\begin{lemma}\label{L:schur}
 Consider the polynomial ring $\Z[x_1,\ldots,x_r;y_1,\ldots,y_{p+r-1}]$.
We have the following algebraic identity:
\begin{equation}\label{E:schur}
 \sum_{j \in [1,r]} \frac{\prod_{i \in [1,p+r-1]}(y_i-x_j)}{\prod_{i\in[1,r]\setminus\{j\}}(x_i-x_j)}
 = \sum_{1 \leq c_1 < \ldots < c_p \leq p+r-1} \hphantom{a} \prod_{i=1}^p (y_{c_i}-x_{c_i-i+1}).
\end{equation}
\end{lemma}

\begin{proof}
We claim that both sides of \eqref{E:schur} are equal to
%\begin{equation}\label{E:expanded_schur}
%\sum_{k=0}^p \left( \sum_{1 \leq c_1 < \ldots < c_k \leq p+r-1}  \hphantom{a}  y_{c_1}\cdots y_{c_k} h_{p-k}(-x_1,\ldots,-x_r)\right),
%\end{equation}
\begin{equation*}
\sum_{k=0}^p  e_k(y_1,\ldots,y_{p+r-1}) h_{p-k}(-x_1,\ldots,-x_r),
\end{equation*}
where $e_k(y_1,\ldots,y_{p+r-1})$ is an elementary symmetric polynomial of degree $k$, and $h_j(-x_1,\ldots,-x_r)$ is a complete homogeneous symmetric polynomial of degree $j$ for $j \geq 0$.  We set $h_j(-x_1,\ldots,-x_r) = 0$ for $j < 0$.

For any positive integer $k$, let $V(x_1,\ldots,x_k)$ denote the Vandermonde determinant $\prod_{1 \leq i < j \leq k}(x_j-x_i)$.
Combining fractions on the left-hand side of \eqref{E:schur}, we get
\begin{align*}
%&\sum_{j \in [1,r]} \frac{\prod_{i \in [1,p+r-1]}(y_i-x_j)}{\prod_{i\in[1,r]\setminus\{j\}}(x_i-x_j)} \\
&\sum_{j \in [1,r]} \frac{(-1)^{j-1}(y_1-x_j)\cdots(y_{p+r-1}-x_j)V(x_1,\ldots,\widehat{x_j},\ldots,x_r)}{V(x_1,\ldots,x_r)}  = \\
&\sum_{\substack{\text{all sequences }\\ 1 \leq c_1 < \ldots < c_k \leq p+r-1 \\ \text{for } 0 \leq k \leq p+r-1}} \left(\sum_{j \in [1,r]} \frac{(y_{c_1}\mathinner{{\cdotp}{\cdotp}{\cdotp}} y_{c_k})(-1)^{j-1}(-x_j)^{p+r-1-k}V(x_1,\mathinner{{\ldotp}{\ldotp}{\ldotp}},\widehat{x_j},\mathinner{{\ldotp}{\ldotp}{\ldotp}},x_r)}{V(x_1,\mathinner{{\ldotp}{\ldotp}{\ldotp}},x_r)}\right).
\end{align*}
The coefficient in $\Z[x_1,\ldots,x_r]$ of the monomial $y_{c_1}\ldots y_{c_k}$ above is equal to
\[(-1)^{p-k}s_{(p-k)}(x_1,\ldots,x_r).\]
%\[ s_{(p-k)}(x_1,\ldots,x_r) := (-1)^{r-1}\frac{a_{(p+r-1-k,r-2,\ldots,1,0)}(x_1,\ldots,x_r)}{V(x_1,\ldots,x_r)},\]
Here,
\[
s_{(p-k)}(x_1,\ldots,x_r) :=  \frac{\det \left[
\begin{matrix}
x_1^{p+r-1-k} & \dots & x_r^{p+r-1-k} \\
x_1^{r-2} & \dots & x_r^{r-2} \\
\vdots & \ddots & \vdots \\
x_1 & \dots & x_r \\
1 & \dots & 1
\end{matrix}
\right]}{(-1)^{r-1}V(x_1,\ldots,x_r)}
\]
\comment{
\[
a_{(p+r-1-k,r-2,\ldots,1,0)}(x_1,\ldots,x_r) :=
\det \left[
\begin{matrix}
x_1^{p+r-1-k} & \dots & x_r^{p+r-1-k} \\
x_1^{r-2} & \dots & x_r^{r-2} \\
\vdots & \ddots & \vdots \\
x_1 & \dots & x_r \\
1 & \dots & 1
\end{matrix}
\right].
\]}
is a Schur polynomial. By the Jacobi-Trudi formula, we have $s_{(p-k)}(x_1,\ldots,x_r)$ $=$ $h_{p-k}(x_1,\ldots,x_r)$ $=$ $(-1)^{p-k}h_{p-k}(-x_1,\ldots,-x_r)$ as desired.

To prove the equality from the right-hand side, we set some more notation.
For any positive integers $b \leq a$, let $\bigbrace{a}{b}$ denote
the set of strictly increasing subsequences $\{c_i\}^{b}_{i=1} \subset [1,a]$.
%Now let  ${p+r-1\brace{p}}$ denote the set of strictly increasing subsequences $\{c_i\}^{p}_{i=1} \subset [1,p+r-1]$.
\comment{
Let  ${r\brack{p}}$ denote the set of \emph{weakly} increasing subsequences $\{c_i\}^{p}_{i=1} \subset [1,r]$.
There is a lexicographic total ordering $<$ on the sets  ${p+r-1\brace{p}}$ and ${r\brack{p}}$, both of which have cardinality $\binom{p+r-1}{p}$.  The unique  order-preserving bijection is given by  $\{c_i\}^{p}_{i=1} \mapsto \{c_i-i+1\}^{p}_{i=1}$.
}
Fix an integer $1 \leq k \leq p-1$ and a sequence $\{f_i\}^{k}_{i=1} \in \bigbrace{p+r-1}{k}$.
Let $\bigbrace{p+r-1}{p}'$ denote the elements of $\bigbrace{p+r-1}{p}$ containing $\{f_i\}^{k}_{i=1}$ as a subsequence.
For any $\{c_i\}^p_{i=1} \in \bigbrace{p+r-1}{p}'$, let $\{\widehat{c}_j\}^{p-k}_{j=1}$ denote
the \emph{weakly} increasing sequence $\{c_i-i+1: c_i \neq f_j \text{ for any } j\}$.
The coefficient in $\Z[x_1,\ldots,x_r]$ of the monomial $y_{f_1}\ldots y_{f_k}$ on the right-hand side of \eqref{E:schur}
is then
\begin{equation}\label{E:monomials}
 \sum_{\{c_i\} \in \bigbrace{p+r-1}{p}'} \left(\prod^{p-k}_{j=1}(-x_{\widehat{c}_j})\right).
\end{equation}
The set $\bigbrace{p+r-1}{p}'$ has cardinality $\binom{p+r-k-1}{p-k}$, exactly the number of degree $p-k$ monomials in $\Z[x_1,\ldots,x_r]$, up to scalar multiples. We claim that no monomial occurs more than once in \eqref{E:monomials}.
Suppose on the contrary that $\{\widehat{c}_j\}^{p-k}_{j=1} = \{\widehat{d}_j\}^{p-k}_{j=1}$
for some pair $\{c_i\}^p_{i=1} \neq \{d_i\}^p_{i=1} \in \bigbrace{p+r-1}{p}'$.
Then there exists a minimum integer $h$ such that $f_h = c_i = d_j$  for some $i \neq j$.
Assuming $i < j$, we have $d_i < d_j = c_i$. Since $d_i \neq f_l$ for any $l \in [1,k]$,
we have $\widehat{c}_{i-h+1} \neq \widehat{d}_{i-h+1}$, a contradiction.
It follows that \eqref{E:monomials} equals $h_{p-k}(-x_1,..., -x_r)$.
\end{proof}

Now we consider the Grassmannian $X = Gr(m,N)$, and compute the structure coefficient $N_{\nu, p}^\nu$ in the equivariant product $[X_\nu]^T\cdot[X_p]^T$ of equivariant Schubert classes. Here    $\sC=\{\sC_1<\cdots<\sC_m\}$ is a general Schubert symbol; $p\in \{1, \cdots, N-m\}$, and $X_p=X_{\sym_p}$ is labeled by   the special  Schubert symbol   $\sym_p=\{N+1-m-p, N+2-m, \cdots, N\}$. We can further assume  $\sC\leq \mathpzc{S}_p$, since  $N_{\nu, p}^\nu$ would vanish otherwise (see e.g. \cite{koku}).
 Let  $I_1 := [1, N-m-p+1]$,  $I_2 := [N-m-p+2, N]$ and $r := \#(I_1\cap \sC)$. It follows from  $\sC\leq \mathpzc{S}_p$ that $r \geq 1$ and $I_2 \setminus \sC$ consists of $p+r-1$ elements.  Write $I_1\cap \sC = \{a_1<\cdots<a_r\}$ and $I_2 \setminus \sC = \{b_1<\cdots<b_{p+r-1}\}$.  We have the following formula of $N_{\nu, p}^\nu$, which gives   the restriction of the special Schubert class $[X_p]^T$ to the $T$-fixed point corresponding to $\sC$.

\begin{lemma}\label{L:restriction_formula2}
The restriction coefficient $N^\sC_{\sC,p}$ is given by
\begin{equation}\label{E:restriction_formula2}
N^\sC_{\sC,p} = \sum_{1\leq c_1<\cdots<c_p\leq {p+r-1}} \hphantom{a} \prod_{i = 1}^p (t_{b_{(c_i)}} - t_{a_{(c_{i}-i+1)}}).
\end{equation}
\end{lemma}

\begin{proof}
Let $\mathcal{A} := [Z_{\nu,\nu}]^T \cdot [\P(E_{N-m-p+1})]^T \in H_T^*(\P^{N-1})$.
By Proposition \ref{P:move_to_projective_space} we have $N^{\nu}_{\nu,p} = \int^T_{\P^{N-1}}\mathcal{A}$.
For $1 \leq j \leq N$, let $\iota^*_j:H_T^*(\P^{N-1}) \to H_T^*\left(\P(\langle \+e_j \rangle)\right)$ denote the restriction to the $T$-fixed point of $\P^{N-1}$ corresponding to the $j$-th basis vector.
Note that
$\iota^*_j \mathcal{A} = 0$ unless $j \in I_1 \cap \nu$, in which case we have
$\iota^*_j \mathcal{A} = \prod_{i \not\in \nu}(t_i-t_j)\prod_{i\in I_2}(t_i-t_j)$.
By the Atiyah-Bott-Berline-Vergne integration formula (see e.g. \cite[\S 2.5]{anderson_notes}), we therefore have
\begin{align*}
\int_{\P^{N-1}}^T\mathcal{A} &= \sum^{N}_{j=1} \frac{\iota^*_j \mathcal{A}}{\prod_{i \in [1,N]\setminus\{j\}}(t_i - t_j)} \\
&=\sum_{j \in I_1 \cap \nu} \frac{\prod_{i \in I_1 \setminus \nu}(t_i-t_j)\prod_{i \in I_2 \setminus \nu}(t_i-t_j)^2\prod_{i \in I_2 \cap \nu}(t_i-t_j)}{\prod_{i\in [1,N] \setminus\{j\}}(t_i-t_j)} \\
&=\sum_{j \in I_1 \cap \nu} \frac{\prod_{i \in I_2 \setminus \nu}(t_i-t_j)}{\prod_{i\in I_1 \cap \nu \setminus\{j\}}(t_i-t_j)} \\
&=\sum_{j \in [1,r]} \frac{\prod_{i \in [1,p+r-1]}(t_{b_i}-t_{a_j})}{\prod_{i\in[1,r]\setminus\{j\}}(t_{a_i}-t_{a_j})}.
\end{align*}
Lemma \ref{L:schur} yields the result, after identifying $y_i$ with $t_{b_i}$ for $1 \leq i \leq p+r-1$, and $x_i$ with $t_{a_i}$ for $1 \leq i \leq r$.
\end{proof}

We observe that Lemma \ref{L:restriction_formula2} is manifestly positive in the sense that the terms $(t_{b_{c_i}} - t_{a_{c_{i}-i+1}})$ are elements of $\Z_{\geq0}[t_2-t_1, \ldots, t_{N}-t_{N-1}]$.
Specializations of these terms yield manifestly positive Pieri rules in types $C$ and $B$ (Theorems \ref{T:type_c_reduction} and \ref{T:type_b_reduction}).
When $N=2n$, the specialization $F_D$ sends all positive roots $t_a-t_b$ (where $a<b$) to positive roots of type $D_n$ except for the simple root $t_n-t_{n+1}$ of type $A_{2n-1}$. Nevertheless,  the following corollary ensures the type $D$ Pieri rule (Theorem \ref{T:type_d_reduction}) is manifestly positive as well.

\begin{cor}\label{C:bad_terms_do_not_occur}
If $N=2n$, then none of the terms $(t_{b_{c_i}} - t_{a_{c_{i}-i+1}})$ in the summation \eqref{E:restriction_formula2} are given by
$(t_{n+1}-t_n)$ unless $I_1 = [1,n]$ and $I_1\cap \sC = \{n\}$.
\end{cor}

\begin{proof}
If the term $(t_{n+1}-t_n)$ occurs, then there exits $1\leq c_1<\cdots < c_p\leq p+r-1$ and $1\leq i\leq p$ such that $t_{b_{(c_i)}} - t_{a_{(c_{i}-i+1)}}=t_{n+1}-t_n$. It follows that $n={a_{(c_{i}-i+1)}}\leq a_r<b_1\leq b_{(c_i)}=n+1$. Hence, we have
 $c_i=1$, $r=1, {b_1} = n+1$, and ${a_1} = n$.
\end{proof}


\begin{thebibliography}{99}

\bibitem{anderson_notes} D. Anderson,\,{\it Introduction to Equivariant Cohomology in Algebraic Geometry}, Contributions to algebraic geometry, 71--92, EMS Ser. Congr. Rep., Eur. Math. Soc., Z{\"u}rich, 2012.

\bibitem{AJS}H. H. Andersen, J. C. Jantzen, and W. Soergel,\,{\it Representations of quantum groups at a $p$th root of unity and of semisimple groups in characteristic $p$: independence of $p$}, Ast\'erisque No. 220, 1994.

\bibitem{Arabia}A. Arabia,\,{\it Cohomologie T-\'equivariante de la vari\'et\'e de drapeaux d'un groupe de Kac-Moody}, Bull. Soc. Math. France 117 (1989), no. 2, 129--165.

%%\bibitem{AnCh}D. Anderson, L. Chen,\,{\it Equivariant quantum Schubert polynomials}, preprint at arXiv: math.AG/1110.5896.
\bibitem{BerSo}N. Bergeron, F. Sottile,\,{\it  A Pieri-type formula for isotropic flag manifolds}, Trans. Amer. Math. Soc. 354 (2002), no. 7, 2659--2705.
 \bibitem{Billey}S. Billey,\,{\it Kostant polynomials and the cohomology ring for $G/B$}, Duke Math. J. 96 (1999), no. 1, 205--224.
  \bibitem{BilleyHai}S. Billey, M.   Haiman,\,{\it  Schubert polynomials for the classical groups}, J. Amer. Math. Soc. 8 (1995), no. 2, 443--482.

 \bibitem{Bour}B. Bourbaki,\,{\it  Lie groups and Lie algebras: Chapters 4--6},  Elements of Mathematics (Berlin). Springer-Verlag, Berlin, 2002.

\bibitem{Brion}M. Brion,\,{\it Equivariant Chow groups for torus actions}, Transform. Groups 2 (1997), no. 3, 225--267.

\bibitem{Buch_calculator} A.S. Buch,\,{\it Equivariant Quantum Calculator}, {A free software package for Maple.}
Available at: http://math.rutgers.edu/~asbuch/equivcalc/


%%\bibitem{buch_Grassmannian}A.S. Buch,\,\textit{Quantum cohomology of Grassmannians}, Compositio Math. 137 (2003), no. 2, 227--235.

%%\bibitem{buch} A.S. Buch,\,{\it Quantum cohomology  of partial flag manifolds}, Trans. Amer. Math. Soc. 357 (2005), no. 2,  443--458.

\bibitem{Buch-equivTwostep} A.S. Buch,\,{\it Mutations of puzzles and equivariant cohomology of two-step flag varieties},
                      Ann. of Math. (2) 182 (2015), 173--220.

%%\bibitem{BKT3}A.S. Buch, A. Kresch and H. Tamvakis,\,\textit{A Giambelli formula for isotropic Grassmannians}, arXiv: math.AG/0811.2781.


\bibitem{BKT4}A.S. Buch, A. Kresch and H. Tamvakis,\,\textit{A Giambelli formula for even orthogonal Grassmannians}, arXiv: math.AG/1109.6669, to appear J. reine angew. Math.

\bibitem{BKT2} A.S. Buch,  A. Kresch  and H. Tamvakis,\,{\it Quantum Pieri rules for isotropic Grassmannians},  {Invent. Math.} 178, no. 2 (2009): 345--405.

\bibitem{buchMihalcea}A.S. Buch, L.C. Mihalcea,\,\textit{Quantum K-theory of Grassmannians}, Duke Math. J. 156 (2011), no. 3, 501--538.

\bibitem{buch_rimanyi}A.S. Buch, R. Rimanyi,\,\textit{Specializations of Grothendieck Polynomials}, C. R. Math. Acad. Sci. Paris 339 (2004), no. 1, 1--4.

\bibitem{buch_ravikumar}A.S. Buch, V. Ravikumar,\,{\it Pieri rules for the K-theory of cominuscule Grassmannians}, J. Reine Angew. Math. 668 (2012), 109--132.


%%\bibitem{CFon}I. Ciocan-Fontanine,\, {\it On quantum cohomology rings of partial flag varieties}, Duke Math. J. 98 (1999), no. 3, 485--524.


\bibitem{fuand}W. Fulton,\, {\it Equivariant cohomology in algebraic geometry}, Eilenberg lectures, Columbia University, Spring 2007. (Notes by D. Anderson.)

%% \bibitem{fupa}W. Fulton, R. Pandharipande,\,{\it Notes on stable maps and quantum cohomology},   Proc. Sympos. Pure Math. 62, Part 2, Amer. Math. Soc., Providence, RI, 1997.

\bibitem{fulton}W. Fulton,\,{\it Young tableaux}, London Mathematical Society Student Texts, vol. 35, Cambridge University Press, Cambridge, 1997.

\bibitem{Fun}A. Fun,\,{\it Raising operators and the Littlewood-Richardson polynomials}, arXiv: math.CO/1203.4729.

  \bibitem{GaSa}L. Gatto, T. Santiago,\,{\it   Equivariant Schubert calculus},  Ark. Mat. 48 (2010), no. 1, 41--55.

 \bibitem{Grah}W. Graham,  \,{\it Positivity in equivariant Schubert calculus}, Duke Math. J. 109 (2001), no. 3, 599--614.

\bibitem{huangli}Y. Huang,   C. Li, \,{\it On equivariant quantum Schubert calculus for $G/P$},    to appear in J. Algebra;  arxiv: math.AG/1506.00872.

%%\bibitem{hum} J.E. Humphreys,\,{\it Introduction to Lie algebras and representation theory}, Graduate Texts in Mathematics 9, Springer-Verlag, New York-Berlin,   1980.
%% \bibitem{humalg} J.E. Humphreys,\,{\it   Linear algebraic groups},    Graduate Texts in Mathematics 21, Springer-Verlag, New York-Berlin,   1975.
\bibitem{Ikeda}T. Ikeda,\,{\it  Schubert classes in the equivariant cohomology of the Lagrangian Grassmannian}, Adv. Math. 215 (2007), no. 1, 1--23.
 \bibitem{IkMa}T. Ikeda, T. Matsumura,\,{\it Pfaffian sum formula for the symplectic Grassmannians}, Math. Zeit. 280 (2015),     269--306.

\bibitem{IMN}T. Ikeda, L. Mihalcea and H. Naruse,\,{\it   Double Schubert polynomials for the classical groups}, Adv. Math. 226 (2011), no. 1, 840--886.
 \bibitem{IkNa}T. Ikeda, H. Naruse,\,{\it Excited Young diagrams and equivariant Schubert calculus}, Trans. Amer. Math. Soc. 361 (2009), no. 10, 5193--5221.

%%\bibitem{kim22}B. Kim,\,{\it On equivariant quantum cohomology},  Internat. Math. Res. Notices 1996, no. 17, 841--851.

%%\bibitem{knutson:noncomplex} A.~Knutson, {\it A Schubert Calculus recurrence from the non-complex $W$-action on $G/B$}, preprint at arxiv: math.CO/0306304.

\bibitem{KnutTao}A. Knutson, T. Tao,\,{\it Puzzles and (equivariant) cohomology of Grassmannians}, Duke Math. J. 119 (2003), no. 2, 221--260.

\bibitem{koku} B. Kostant, S. Kumar,\,{\it The nil Hecke ring and the cohomology of $G/P$ for a Kac-Moody group $G$}, Adv. in Math. 62 (1986), 187-237.


\bibitem{Krei}V. Kreiman,\,{\it Equivariant Littlewood-Richardson skew tableaux}, Trans. Amer. Math. Soc. 362 (2010), no. 5, 2589--2617.

\bibitem{kumar} S. Kumar,\,{\it Kac-Moody groups, their flag varieties and representation theory}, Progress in Mathematics 204, Birh\"auser Boston, Inc., Boston, MA, 2002.


\bibitem{LaRaSa}V. Lakshmibai, K.N. Raghavan and P. Sankaran,\,{\it Equivariant Giambelli and determinantal restriction formulas for the Grassmannian}, Pure Appl. Math. Q. 2 (2006), no. 3, Special Issue: In honor of Robert D. MacPherson. Part 1, 699--717.


\bibitem{Laksov}D. Laksov,\,{\it Schubert calculus and equivariant cohomology of Grassmannians}, Adv. Math. 217 (2008), no. 4, 1869--1888.
%%\bibitem{lamshi} T. Lam, M. Shimozono,\,{\it Quantum cohomology of $G/P$ and homology of affine Grassmannian}, Acta Math. 204 (2010), no. 1, 49--90.

\bibitem{lamshi22} T. Lam, M. Shimozono,\,{\it  Equivariant Pieri Rule for the homology of the affine Grassmannian}, J. Algebraic Combin. 36 (2012), no. 4, 623--648.
\bibitem{lamshi33}T. Lam, M. Shimozono,\,{\it  Quantum double Schubert polynomials represent Schubert classes}, Proc. Amer. Math. Soc. 142 (2014), no. 3, 835--850.
%%\bibitem{leungli33} N.C. Leung,  C. Li,\,{\it Functorial relationships between $QH^*(G/B)$ and $QH^*(G/P)$},  J. Differential Geom. 86 (2010), no. 2, 303--354.
%%\bibitem{leungliQtoC}N.C. Leung, C. Li, \,{\it Classical aspects of quantum cohomology of generalized flag varieties},   Int. Math. Res. Not. IMRN 2012, no. 16, 3706--3722.
%% \bibitem{leungli22} N.C. Leung,  C. Li,\,{\it Gromov-Witten invariants for $G/B$ and Pontryagin product for $\Omega K$}, Trans. Amer. Math. Soc. 364 (2012), no. 5, 2567--2599.

\bibitem{leungli44}N.C. Leung,   C. Li, \,{\it Quantum Pieri rules for tautological subbundles},    {Adv. Math.}   248 (2013), 279--307.


%%\bibitem{Macdonald}I.G.  Macdonald,\,{\it Symmetric functions and Hall polynomials. With contributions by A. Zelevinsky},   Oxford University Press, New York, 2nd edition,  1995.
%%\bibitem{mih2}L.C. Mihalcea,\,{\it Positivity in equivariant quantum Schubert calculus}, Amer. J. Math. 128 (2006), no. 3, 787--803.
%%\bibitem{mih}L.C. Mihalcea,\,{\it Equivariant quantum Schubert calculus}, Adv. Math. 203 (2006), no. 1, 1--33.

\bibitem{mih} L.C. Mihalcea,\,{\it Equivariant quantum cohomology of homogeneous spaces},  Duke Math. J. 140 (2007), no. 2, 321--350.

 \bibitem{Mihalcea_Giambelli}L.C. Mihalcea,\,{\it Giambelli formulae for the equivariant quantum cohomology of the Grassmannian}, Trans. Amer. Math. Soc. 360 (2008), no. 5, 2285--2301.

\bibitem{Molev}A.I. Molev,\,{\it Littlewood-Richardson polynomials}, J. Algebra 321 (2009), no. 11, 3450--3468.

\bibitem{MoSa-Pieri}A.I. Molev,   B.E. Sagan,\,{\it  A Pieri rule for generalized factorial Schur functions},   Proceedings of the 9-th Conference on Factorial Power Series and Algebraic Combinatorics,
Vienna,  1997, Vol. 3, 517--523.
\bibitem{MoSa}A.I. Molev, B.E. Sagan,\,{\it A Littlewood-Richardson rule for factorial Schur functions}, Trans. Amer. Math. Soc. 351 (1999), no. 11, 4429--4443.
%%\bibitem{peterson} D. Peterson,\,{\it Quantum cohomology of $G/P$}, Lecture notes at MIT, 1997 (notes by J. Lu and K. Rietsch).

\bibitem{Peterson}D. Peterson, lectures, (1997).


%%\bibitem{pragrat}P. Pragacz, J. Ratajski,\,\textit{A Pieri-type theorem for Lagrangian and odd orthogonal Grassmannians}, J. Reine Angew. Math. 476 (1996), 143--189.


\bibitem{Ravi}V. Ravikumar,\,{\it Triple Intersection Formulas for Isotropic Grassmannians}, to appear in Algebra and Number Theory; preprint at arxiv: math.AG/1403.1741.

\bibitem{rob}S. Robinson,\,{\it A Pieri-type formula for $H_T^*(SL_n(\mathbb{C})/B)$}, J. Algebra 249 (2002), no. 1, 38--58.


%%\bibitem{sietian}B. Siebert, G. Tian,\,\textit{On quantum cohomology rings of Fano manifolds and a formula of Vafa and Intriligator},  Asian J. Math. 1 (1997), no. 4, 679--695.

\bibitem{Sott}F. Sottile,\,{\it Pieri-type formulas for maximal isotropic Grassmannians via triple intersections},   Colloq. Math. 82 (1999), no. 1, 49--63.
\bibitem{Tamv} H. Tamvakis, \,{\it  Giambelli and degeneracy locus formulas for classical $G/P$ spaces}, preprint at arxiv: math.AG/1305.3543.
\bibitem{TamvWilson}H. Tamvakis, E. Wilson,\,{\it  Double theta polynomials and equivariant Giambelli formulas},  preprint at arxiv: math.AG/1410.8329.




\bibitem{ThYo}H. Thomas, A. Yong,\,{\it Equivariant Schubert calculus and jeu de taquin}, to appear in  Annales de l'Institut Fourier;  preprint at arxiv: math.CO/1207.3209.

\bibitem{Wilson} E.V. Wilson, \,{\it Equivariant Giambelli Formulae for Grassmannians}, Ph.D. Thesis, University of Maryland, College Park. 2010.
%%\bibitem{wo}C.T. Woodward,\,{\it On D. Peterson's comparison formula for Gromov-Witten invariants of $G/P$},  Proc. Amer. Math. Soc. 133 (2005), no. 6,   1601--1609.%;  math.AG/0206073.


%%\bibitem{wo}C.T. Woodward,\,{\it On D. Peterson's comparison formula for Gromov-Witten invariants of $G/P$},  Proc. Amer. Math. Soc. 133 (2005), no. 6,   1601--1609.%;  math.AG/0206073.



\end{thebibliography}
\end{document}